\documentclass[a4paper,11pt]{scrartcl}

\usepackage[english]{babel}
\usepackage[utf8]{inputenc}
\usepackage{ulem} 

\usepackage{authblk}

\usepackage{hyperref}

\usepackage{tikz}
\usetikzlibrary{intersections,positioning}
\usetikzlibrary{intersections, quotes, angles}
\usepackage{color}

\usepackage{amssymb}
\usepackage{amsmath}
\usepackage{amsthm}
\usepackage{bbm}

\usepackage{mathtools}
\usepackage{nicefrac}
\usepackage{comment}
\usepackage{enumitem}
\usepackage{marginnote}

\KOMAoptions{abstract=on} 

\theoremstyle{plain}
\newtheorem{theorem}{Theorem}[section]
\newtheorem{lemma}[theorem]{Lemma}
\newtheorem{prop}[theorem]{Proposition}
\newtheorem{corolla}[theorem]{Corollary}
\newtheorem{probl}[theorem]{Open Problem}
\newtheorem*{que*}{Question}

\theoremstyle{definition}
\newtheorem{defi}[theorem]{Definition}

\newtheorem{nota}[theorem]{Notation}

\theoremstyle{remark}
\newtheorem{rem}[theorem]{Remark}
\newtheorem{es}[theorem]{Example}

\numberwithin{equation}{section} 

\newcommand{\C}{\mathbb{C}}

\newcommand{\K}{\mathbb{K}}

\newcommand{\N}{\mathbb{N}}

\newcommand{\Q}{\mathbb{Q}}
\newcommand{\R}{\mathbb{R}}

\newcommand{\caK}{\mathcal{K}}
\newcommand{\caL}{\mathcal{L}}

\newcommand{\uC}{\mathrm{C}}

\newcommand{\uL}{\mathrm{L}}

\newcommand{\ue}{\mathrm{e}}

\newcommand{\ui}{\mathrm{i}}

\makeatletter
\newcommand{\allmodesymb}[2]{\relax\ifmmode{\mathchoice
{\mbox{\fontsize{\tf@size}{\tf@size}#1{#2}}}
{\mbox{\fontsize{\tf@size}{\tf@size}#1{#2}}}
{\mbox{\fontsize{\sf@size}{\sf@size}#1{#2}}}
{\mbox{\fontsize{\ssf@size}{\ssf@size}#1{#2}}}}
\else
\mbox{#1{#2}}\fi}
\makeatother

\renewcommand{\subset}{\subseteq}

\newcommand{\de}{\,\mathrm{d}}
\newcommand{\st}{\,:\,}
\newcommand{\sto}[1]{\xrightarrow[\,\, #1 \,\,]{}}

\newcommand{\rpl}{\mathbb{R}_{\ge 0}}
\newcommand{\rplc}{\overline{\mathbb{R}}_{\ge 0}}

\newcommand{\vertiii}[1]{{\left\vert\kern-0.25ex\left\vert\kern-0.25ex\left\vert #1 \right\vert\kern-0.25ex\right\vert\kern-0.25ex\right\vert}}
\newcommand{\abs}[1]{\lvert#1\rvert}

\newcommand{\norm}[1]{\lVert#1\rVert}
\newcommand{\nnorm}[1]{\big\lVert#1\big\rVert}
\newcommand{\Norm}[1]{\Big\lVert#1\Big\rVert}
\newcommand{\normm}[1]{\bigg\lVert#1\bigg\rVert}

\makeatletter
\DeclareRobustCommand\widecheck[1]{{\mathpalette\@widecheck{#1}}}
\def\@widecheck#1#2{%
    \setbox\z@\hbox{\m@th$#1#2$}%
    \setbox\tw@\hbox{\m@th$#1%
       \widehat{%
          \vrule\@width\z@\@height\ht\z@
          \vrule\@height\z@\@width\wd\z@}$}%
    \dp\tw@-\ht\z@
    \@tempdima\ht\z@ \advance\@tempdima2\ht\tw@ \divide\@tempdima\thr@@
    \setbox\tw@\hbox{%
       \raise\@tempdima\hbox{\scalebox{1}[-1]{\lower\@tempdima\box
\tw@}}}%
    {\ooalign{\box\tw@ \cr \box\z@}}}
\makeatother

\DeclareMathOperator{\dom}{dom}

\DeclareMathOperator{\re}{Re}          
\DeclareMathOperator{\im}{Im}          
\DeclareMathOperator{\lip}{Lip}

\DeclareMathOperator{\interior}{int}

\DeclareMathOperator{\supp}{supp}

\begin{document}

\normalem 

\title{On Eventual Regularity Properties of Operator-Valued Functions}
\author{Marco Peruzzetto\footnote{\textit{Email address:} \texttt{peruzzetto@math.uni-kiel.de}}}
\affil{Christian-Albrechts-Universität zu Kiel}
\date{} 
\begin{titlepage}
\maketitle
\vspace{-1.5cm}
\begin{center}
\emph{Dedicated to Rainer Nagel on the Occasion of His 80\textsuperscript{th} Birthday}
\end{center}
\vspace{1cm}
\begin{abstract}
Let $\caL(X;Y)$ be the space of bounded linear operators from a Banach space $X$ to a Banach space $Y$. Given an operator-valued function $u:\R_{\geq 0}\rightarrow \caL(X;Y)$, suppose that every orbit $t\mapsto u(t)x$ has a regularity property (e.g.\ continuity, differentiability, etc.) on some interval $(t_x,\infty)$ in general depending on $x\in X$.
In this paper we develop an abstract set-up based on Baire-type arguments which allows, under certain conditions, removing the dependency on $x$ systematically.

Afterwards, we apply this theoretical framework to several different regularity properties that are of interest also in semigroup theory. In particular, a generalisation of the prior results on eventual differentiability of strongly continuous functions $u:\R_{\geq 0}\rightarrow \caL(X;Y)$ obtained by Iley and Bárta follows as a special case of our method.
\end{abstract}
\end{titlepage}


\section{Introduction}
Let $\caL(X;Y)$ be the space of bounded linear operators from a Banach space $X$ to a Banach space $Y$ and let $u:\R_{\ge 0}\rightarrow\caL(X;Y)$ be an operator-valued function defined on the positive real axis, which we imagine representing the time. \\
Suppose $\mathtt{P}$ is a property that a vector-valued function defined on an interval may have or not. 
We think of $\mathtt{P}$ as a regularity property like boundedness, continuity, differentiability, etc.. Assume that for each $x\in X$ there is a time $t_x\ge 0$ such that the vector-valued orbit $t\mapsto u(t)x$ has $\mathtt{P}$ on the interval $(t_x,\infty)$. In this paper we find sufficient conditions on $\mathtt{P}$ in order to guarantee the existence of a ``universal'' moment $t\ge 0$, independent of $x$, such that for every $x\in X$ the orbits $t\mapsto u(t)x$ have $\mathtt{P}$ on $(t,\infty)$.
If this is possible, we will say that ``individual eventual regularity'' implies ``uniform eventual regularity''.

Such a question arose originally from semigroup theory, where situations of this type appear in many cases. In this context, a prominent example for $\mathtt{P}$ is differentiability. 
In the literature, see e.g.\ \cite[Chapter~II.4.b]{EnNa}, a $C_0$-semigroup $T$ is said to be \textit{eventually differentiable} if the equivalent conditions (i)--(iii) of the following elementary proposition (see also Lemma 4.2 and Corollary 4.3 in Chapter 2 of \cite{pazy}) are satisfied for some $t_0\ge 0$.

\begin{prop}\label{prop pazy}
Let $T= \big(T(t)\big)_{t\ge 0}$ be a $C_0$-semigroup on a Banach space $X$ with generator $A$ having domain $\dom(A)\subset X$. Then, the following assertions are equivalent for $t_0\ge 0$:
\begin{enumerate}[label=(\roman*), font=\normalfont, noitemsep]
\item For each $x\in X$, the orbit $t\mapsto T(t)x$ is differentiable on $(t_0,\infty)$;
\item For each $t>t_0$, the inclusion $T(t) X\subset \dom(A)$ is satisfied;
\item The semigroup $T$ is continuous on $(t_0,\infty)$ with respect to the operator norm.
\end{enumerate}
Moreover, if \textup{(i)--(iii)} are satisfied, then the following assertion holds true for each natural number $k\ge 1$:
\begin{enumerate}[font=\normalfont, noitemsep]
\item[(iv)] The semigroup $T$ is $k$-times differentiable on $\big((k+1)t_0,\infty\big)$ with respect to the operator norm.
\end{enumerate}
\end{prop}

\noindent Later, we will call eventually differentiable semigroups also \textit{uniformly eventually differentiable} in order to highlight the independence from $x$ of the starting time of the differentiability of the orbits.
In the 1968 paper \cite{pazya}, Pazy used assertion (iv) of Proposition \ref{prop pazy} to characterize eventually differentiable $C_0$-semigroups through the spectral properties of their generators (see \cite[Theorem~4.7 and Theorem~4.8]{pazy} and \cite[Chapter II, Theorem~4.14]{EnNa}). His result has made it possible to find several non-trivial examples of eventually differentiable semigroups, like semigroups for delay differential equations or Volterra equations, which has been studied extensively e.g.\ in \cite[Chapter~VI.6]{EnNa}, \cite{batty}, \cite{batty1} and \cite{barta1}. \\
In 2007 Batty proposed an alternative way for characterizing eventual differentiability of a $C_0$-semigroup $T$ by formulating the following question.

\begin{que*}[{\cite[Open~Question~2]{batty}}]
Suppose that, for each $x\in X$, there exists $t_x\ge 0$ such that the orbit $t\mapsto T(t)x$ is differentiable on $(t_x,\infty)$. Is the semigroup $T$ eventually differentiable?
\end{que*}

Observe that the advantage of this approach, which we pursue in this article, is that it does not make use of any specific tool like the generator and therefore it lends itself to be applied also to more general operator families than merely $C_0$-semigroups. \\
Batty's question was answered positively by Iley in \cite[Proposition~2.2]{iley} and subsequently generalized by Bárta in \cite{barta} to arbitrary strongly continuous mappings. However, in contrast to Iley, who employed a Baire-type argument, the proof of Bárta is built on a gliding hump argument reminiscent of Hahn's proof of the uniform bounded principle. The question, whether a parallel proof relying on a Baire-type argument exists, arose naturally and marked the starting point for this article.

Here, we not only obtain a further generalisation of Bárta's theorem (see Proposition \ref{differentiab} in Section \ref{section differentiability}), but we also provide an abstract set-up into which all known results can be immersed. We consequently develop a new method for tackling the problem of the equivalence between individual and uniform eventual regularity. This allows us to move beyond differentiability, and to consider many other situations to which this method can be applied.

\subsection*{Structure of the article}
We start Section 2 by making clear what it means for a vector-valued function to have a regularity property $\mathtt{P}$ individually eventually and uniformly eventually (see Definition \ref{definition of P}). Afterwards different examples and counterexamples which indicate how linearity represents a sort of watershed for problems of this kind are provided: the expectation is that one has equivalence between individual regularity and uniform regularity if and only if the considered property $\mathtt{P}$ is linear (see Definition \ref{defi property linear}). 

In the second part of Section 2 we investigate subspaces of a topological vector space that can be associated with a family of extended seminorms in respect to which they satisfy particular conditions of finiteness and convergence. We therefore introduce the new concept of $P$-\textit{subspaces} (see Definition \ref{definition P subspace}) and prove the main abstract result of this paper (Proposition \ref{prop baire}), consisting of a Baire-type theorem specifically for $P$-subspaces.

Finally, in Section \ref{section applications} the abstract tools from the previous section are applied; the key insight here is that many regularity properties can be rephrased in terms of union of $P$-subspaces. As a result, we obtain that uniform eventual regularity is implied by individual eventual regularity whenever for ``regularity'' we take e.g.\ one of the following notions:
\begin{enumerate}[label=(\alph*), font=\normalfont, noitemsep]
\item local and global boundedness;
\item local and global $\alpha$-Hölder continuity, continuity, uniform continuity, local bounded variation, local and global absolute continuity;
\item $k$-differentiability and smoothness;
\item real analyticity, with further considerations when the semigroup property is present.
\end{enumerate}


\subsection*{Notation and Preliminaries}

For a number $a\in\R$ we write $\R_{\ge a}\coloneqq [a,\infty)$ and $\R_{>a}\coloneqq (a,\infty)$, and we use the notation $\rplc\coloneqq \rpl\cup\{+\infty\}$ for the two-point compactification of $\rpl$. \\
Given a set $S$, a function $f:S\rightarrow\rplc$ and $a\in\rplc$, we put
\[
[f\le a]\coloneqq \big\{x\in S\st f(x)\in [0,a]\big\} \quad\mbox{and}\quad [f < a]\coloneqq \big\{x\in S\st f(x)\in [0,a)\big\}.
\]
We write $\N$ for the positive natural numbers $\{1,2,3,\ldots\}$ and $\N_0\coloneqq\N\cup\{0\}$.

If not otherwise specified, $X$ and $Y$ are Banach spaces over $\R$ or $\C$ and $\caL(X;Y)$ is the space of bounded linear operators from $X$ to $Y$. We also use the abbreviation $\caL(X)\coloneqq \caL(X;X)$. For a function 
\[
u:\rpl\rightarrow\caL(X;Y)
\]
and a point $x\in X$ we denote by $u_x$ the \textit{orbit} of $x$ under $u$, that is, the mapping 
\[
u_x:\rpl\rightarrow Y,\quad u_x(t)\coloneqq u(t) x.
\]

Let $Z$ be a topological space and let $A\subset Z$. We recall that the set $A$ is called \textit{nowhere dense} if the interior of its closure is empty, in symbols $\interior(\overline{A})=\emptyset$. The set $A$ is called \textit{meagre} (or of \textit{first category)} if it is a countable union of nowhere-dense sets. If $A$ is not meagre in $Z$, then it is of \textit{second category} in $Z$. The space $Z$ is a \textit{Baire space} if every non-empty open set is of second category in $Z$. A Baire space cannot be written as countable union of meagre sets. Moreover, a normed vector space is a Baire space if and only if it is of second category in itself. 

A general form of the Baire category theorem asserts that any complete pseudo-metric space and any locally compact regular space is a Baire space (cf. \cite[Theorem~34]{kelley}),
but for the purpose of this article it suffices to use the following simple version.

\begin{theorem}[Baire category theorem -- simple version]\label{meagre}
${}$ \\
Every non-empty complete metric space is a Baire space.
\end{theorem}

\section{The abstract framework}
The goal of this paragraph is to introduce the notion of $P$-subspaces together with their main features, and to establish a connection with linear regularity properties. 

\subsection{Regularity properties}

The following definition is inspired by Definition 4.1 in \cite{dgk}. 

\begin{defi}\label{definition of P}
Let $u:\rpl\rightarrow\caL(X;Y)$ be a function and let $\mathtt{P}$ be a regularity property such that, for every $x\in X$, it is sensible to assert that the orbit $u_x$ has $\mathtt{P}$ on an interval $I\subset\rpl$ (see also Remark \ref{remark definition}). We say that $u$ has
\begin{enumerate}[label=(\alph*), font=\normalfont]
\item \textit{individually eventually} the property $\mathtt{P}$ if for each $x\in X$ there is $t_x\in\rpl$ such that the orbit $u_x$ has the property $\mathtt{P}$ on the interval $(t_x,\infty)$;
\item \textit{uniformly eventually} the property $\mathtt{P}$ if there is $t\in\rpl$ such that for all $x\in X$ the orbit $u_x$ has the property $\mathtt{P}$ on the interval $(t,\infty)$.
\end{enumerate}
\end{defi}

\begin{rem}\label{remark definition}
Regularity properties that are pointwise defined can also be considered as properties on an interval: in this situation, it suffices to interpret the assertion ``$u_x$ has $\mathtt{P}$ on $I$'' as ``$u_x$ has $\mathtt{P}$ at each point $t\in I$''. This is for example the case in differentiability, because a function is, by definition, differentiable on an interval $I$ if it is differentiable at each point $t\in I$.\\
We point out that for global properties, like boundedness, the situation is of course different: asserting that a function is bounded on $I$ is not the same as saying that it is bounded at every point $t\in I$.
\end{rem}

While (b) always implies (a), our aim is to find conditions on the regularity property $\mathtt{P}$ in order to guarantee that also the converse implication holds true. By defining for each $n\in\N$ the set 
\begin{align}\label{subsets En}
E_n\coloneqq \big\{x\in X\st u_x\mbox{ has the property } \mathtt{P}\mbox{ on } \R_{> n} \big\},
\end{align}
we observe that (a) just means that 
\begin{align}\label{union}
\bigcup_{n\in\N} E_n=X,
\end{align}
while (b) asserts  that there is some $n_0\in\N$ such that $E_{n_0}=X$. 

In the context of Remark \ref{remark definition}, it is useful to bear in mind that, for pointwise properties, each set $E_n$ can be rewritten as 
\[
E_n=\bigcap_{s > n} \big\{x\in X\st u_x\mbox{ has the property } \mathtt{P}\mbox{ at } s \big\}.
\]

\begin{defi}\label{defi property linear}
We call a regularity property $\mathtt{P}$ \textit{linear} if the set $E_n$ defined in \eqref{subsets En} is a linear subspace of $X$ for each $n\in\N$.
\end{defi}
  
If we give up linearity, it is not difficult to find properties $\mathtt{P}$ such that (a) and (b) are not equivalent. Moreover, in this case not even the structure of $X$ helps, as Example \ref{counterexample} shows, where in (i) $X$ is finite-dimensional and in (ii) $X$ is a Banach lattice. 

\begin{es}\label{counterexample}
${}$
\begin{enumerate}[label=(\roman*), font=\normalfont]
\item On the Banach space $\R^2$ consider the operator-valued function
\[
u:\rpl\rightarrow\caL(\R^2), \qquad u(t)\coloneqq 
\left(
\begin{array}{cc}
t^2 & t \\ 0 & 0.
\end{array}
\right)
\]
and the regularity property $\mathtt{P}$, defined as follows:
\[
u_x \mbox{ has } \mathtt{P}\mbox{ on } I \quad :\Leftrightarrow\quad u_x \mbox{ is injective or identically zero on } I,
\]
where $I\subset\rpl$ is an interval and $x\in\R^2$.
It is easy to see that $u$ has individually eventually the property $\mathtt{P}$ but not uniformly. This property is clearly not linear.
\item Suppose that $X$ and $Y$ are Banach lattices. For $u:\rpl\rightarrow\caL(X;Y)$ consider the regularity property $\mathtt{P}$ defined by
\[
u_x \mbox{ has } \mathtt{P}\mbox{ at } t \quad :\Leftrightarrow\quad u_x(t)\ge 0 \mbox{ whenever } x\ge 0,
\]
where for $t\in\rpl$ and $x\in X$.
Here, $\mathtt{P}$ should be interpreted as a pointwise property, see also Remark \ref{remark definition}. Observe that $\mathtt{P}$, called \textit{positivity}, is in general not linear. Moreover, it is possible to construct $C_0$-semigroups which are eventually positive individually, but not uniformly; see \cite[Example~5.7 and Example~5.8]{dgk}.
\end{enumerate}
\end{es}

On the contrary, if $\mathtt{P}$ is linear, then (a) and (b) are always equivalent whenever $X$ is finite-dimensional and, even if $X$ is infinite-dimensional, it seems to us to be difficult to find a counterexample. 
The reason lies in the following facts. If (a) is satisfied i.e., \eqref{union} holds true, then the Baire category theorem (Theorem \ref{meagre}) guarantees that there is some $n_0\in\N$ such that $E_{n_0}$ is of second category. Now, if (b) is not satisfied, then $E_{n_0}\neq X$. Hence, in order to find a counterexample, one needs to handle with a proper subspace of second category, which is therefore necessarily incomplete. Actually, it is easy to show that $X$ has proper subspaces of second category: thanks to the Kuratowski-Zorn lemma one proves the existence of a discontinuous functional, whose kernel is of this kind. Unfortunately, as far as we know, until nowadays there are only existence proofs and no \textit{concrete examples} of a second category incomplete normed vector space. The available information does not appear to us to be useful for constructing a counterexample. We consequently formulate the following open question.

\begin{probl}
Do there exist Banach spaces $X,Y$, a function $u:\rpl\rightarrow\caL(X;Y)$ and a linear property $\mathtt{P}$ such that (a) and (b) from Definition \ref{definition of P} are not equivalent?
\end{probl}

In order to show the equivalence of (a) and (b), the Baire category theorem turns out to be very useful. It can be applied if one shows that each subspace $E_n$ is either equal to $X$ or it is meagre. This happens for example, if all spaces $E_n$ are closed, which is the case of many usual properties (as in Example \ref{example}).

\begin{es}\label{example}
${}$
\begin{enumerate}[label=(\roman*), font=\normalfont]
\item An operator-valued function $u:\rpl\rightarrow\caL(X;Y)$ is uniformly eventually zero if there exists $t_0\in\rpl$ such that $u(t)=0$ for all $t\ge t_0$ (a uniformly eventually zero $C_0$-semigroup is also called \textit{nilpotent}; see \cite[Chapter~I.4.17]{EnNa}). This is equivalent to saying that $u$ has uniformly eventually the (pointwise defined) property $\mathtt{P}$ where,
\[
u_x \mbox{ has property } \mathtt{P}\mbox{ at } t \quad :\Leftrightarrow\quad u_x(t)=0.
\]
for $x\in X$ and $t\in\rpl$.
Using the Baire category theorem one can easily prove that $u$ is uniformly eventually zero if it is individually eventually zero. Indeed, put
\[
E_n\coloneqq\big\{x\in X\st u_x(t)=0 \mbox{ whenever } t\in\R_{> n}\big\}=\bigcap_{t > n}\ker u(t).
\]
Clearly, each vector space $E_n$ is closed. Thus, as $\bigcup_{n\in\N} E_n =X$, it follows from the Baire category theorem that there is some $n_0\in\N$ with $\interior(E_{n_0})\neq \emptyset$, and therefore $E_{n_0}=X$, which is the claim.
\item An operator-valued function $u:\rpl\rightarrow\caL(X;Y)$ is individually eventually strongly measurable if and only if it is uniformly eventually strongly measurable. Indeed, for all $n\in\N$ put
\[
E_n\coloneqq\big\{x\in X\st u_x \mbox{ is strongly measurable on } \R_{>n}\big\}.
\]
By \cite[Theorem~3.5.4.(3)]{HiPh}, each vector space $E_n$ is closed. Again, as $\bigcup_{n\in\N} E_n=X$, the Baire category theorem assures that there is $n_0\in\N$ with $E_{n_0}=X$, which shows the claim.
\item Given an operator-valued function $u:\rpl\rightarrow\caL\big(\caL(X;Y)\big)$, let the property $\mathtt{P}$ be defined in the following way:
\[
u_T \mbox{ has property } \mathtt{P}\mbox{ at } t \quad :\Leftrightarrow\quad u_T(t)\mbox{ is a compact operator}
\]
for $t\in\rpl$ and $T\in\caL(X;Y)$.
Clearly, $\mathtt{P}$ is linear, and $u$ has individually eventually the property $\mathtt{P}$ if and only if it has $\mathtt{P}$ uniformly eventually. Indeed, the vector spaces 
\[
E_n\coloneqq\big\{T\in\caL(X;Y)\st u_{T}(t)\mbox{ is compact for each } t\in\R_{> n}\big\}
\]
are closed and $\bigcup_{n\in\N} E_n= X$. Thus, the Baire category theorem can be applied and yields $E_n=X$ for each $n\ge n_0$ for some $n_0\in\N$, which is the claim. \\
Alternatively, (iii) can be obtained as a special case of (i), by passing to the quotient space $\caL(X;Y)/\caK(X;Y)$, where $\caK(X;Y)$ denotes the subspace of $\caL(X;Y)$ consisting of the compact operators.
\item The following example is similar to (iii). Consider an operator-valued function $u:\rpl\rightarrow\caL\big(\caL(X;Y)\big)$ and let the property $\mathtt{P}$ be defined in the following way: 
\[
u_T \mbox{ has property } \mathtt{P}\mbox{ at } t \quad :\Leftrightarrow\quad u_T(t)\mbox{ is a finite-rank operator}
\]
for $t\in\rpl$ and $T\in\caL(X;Y)$.
Then, $\mathtt{P}$ is linear. Moreover, the vector spaces
\[
E_n\coloneqq \big\{T\in\caL(X;Y)\st u_{T}(t)\mbox{ has finite rank for each } t\in\R_{> n}\big\}
\]
are meagre if they are not equal to $\caL(X;Y)$. Indeed, suppose $E_n\neq \caL(X;Y)$ for some $n\in\N$. As $E_n=\bigcap_{t > n} G_t$, where 
\[
G_t\coloneqq \big\{T\in\caL(X;Y)\st u_{T}(t)\mbox{ has finite rank}\big\},
\]
there exists $t_0\in\R_{> n}$ with $G_{t_0}\neq \caL(X;Y)$. As every subset of a meagre set is meagre, it suffices to show that $G_{t_0}$ is meagre. Observe that $G_{t_0}=\bigcup_{k\in\N} H_k$, where
\[
H_k\coloneqq \big\{T\in\caL(X;Y)\st u_T(t_0) \mbox{ has rank less or equal to } k\big\}
\]
As the subspace $G_{t_0}$ is assumed to be proper, it has empty interior. On the other hand, $H_k$ is closed and contained in $G_{t_0}$ for each $k\in\N$. Hence, $H_k$ is nowhere-dense and $G_{t_0}$ is meagre. Thus, by the Baire category theorem $u$ has individually eventually the property $\mathtt{P}$ if and only if it has $\mathtt{P}$ uniformly eventually.
\end{enumerate}
\end{es}

Example \ref{example} shows that closedness is particularly convenient for applying the Baire category theorem. However, there are other interesting linear properties (like those in Section \ref{section applications}) such that the correspondent vector spaces $E_n$, when proper, are not necessarily closed or neither contained in a proper closed subspace. Nevertheless, it is still possible to show that each of these vector spaces is meagre if closedness can be replaced suitably: for this reason, we introduce now a new class of subspaces of a topological vector space. Later, in Section \ref{section applications}, we provide several examples of subspaces belonging to this class.

\subsection{Extended seminorms and \texorpdfstring{$P$}{P}-subspaces}

In the most part of the article, extended seminorms will play a prominent role.
A function $p:V\rightarrow\rplc$ on a vector space $V$ over $\K\in\{\R, \C\}$ is called an \textit{extended seminorm} on $V$ if it satisfies 
\begin{enumerate}[label=(\alph*), font=\normalfont]
\item $p(\lambda x)=\abs{\lambda}p(x)$ for all $\lambda\in\K$ and $x\in V$;
\item $p(x+y)\le p(x)+p(y)$ for all $x,y\in V$;
\end{enumerate}
where we adopt the convention $0\cdot\infty=0$.

Vector spaces with norms that may attain infinite values were introduced and studied by Beer in \cite{beer}. The main motivation of that paper was to develop new functional analytic tools in order to achieve a better understanding of structures like vector spaces of continuous functions defined on a non-compact space. His work was carried on by Salas and Tapia-García, who gave in \cite{salasgarcia} the definition of extended seminorms for the first time and studied thoroughly extended topological vector spaces and extended locally convex spaces. In both \cite{beer} and \cite{salasgarcia} Baire-type arguments are used; unfortunately they do not seem to be useful in the context of the present article. This is the content of the following remark.

\begin{rem}
Let $V$ be a vector space and let $\norm{\cdot}$ be an extended norm on $V$, i.e., an extended seminorm with the property that $\norm{x}=0$ if and only if $x=0$ ($x\in V$). By setting $V_{\mathrm{fin}}\coloneqq \big[\norm{\cdot}<\infty\big]$, Beer proved that $(V,\norm{\cdot})$ is complete if and only if $(V_{\mathrm{fin}}, \norm{\cdot}_{\vert V_{\mathrm{fin}}})$ is a Banach space; see \cite[Proposition~3.11]{beer} for more details. In this case $V$ is also called \textit{extended Banach space}. 

It is worthy to note that every extended Banach space is also a Baire space. Indeed, if $(V,\norm{\cdot})$ is an extended Banach space, then 
\[
d:V\times V\rightarrow\rpl, \qquad d(x,y)\coloneqq \min\big\{1, \norm{x-y}\big\}
\]
is a metric on $V$ having the same Cauchy sequences and inducing the same topology as $\norm{\cdot}$ (see e.g.\ \cite[Proposition~4.3.8]{engelking}). Therefore, the Baire category theorem can be applied to $V$. 
At this point, it is questionable whether a line of reasoning like that in Example \ref{example} can be pursued also for extended Banach spaces. For example, let $(X,\norm{\cdot}_X)$ be a Banach space and let $u:\rpl\rightarrow\caL(X)$ be a strongly continuous function. One can show that 
\[
\norm{\cdot}:X\rightarrow\rplc, \qquad \norm{x}\coloneqq \norm{u_x}_\infty+\sup_{0<s<t}\Norm{\frac{u_x(t)-u_x(s)}{t-s}}_X
\]
is a complete extended norm on $X$ with $X_{\mathrm{fin}}=\lip(\rpl;X)$, the Banach space of $X$-valued Lipschitz functions, and the subspaces $E_n$ defined in \eqref{spaces differentiabiliy} are $\norm{\cdot}$-closed. Almost the same idea has already been used for \cite[Theorem~4.2]{beerhoffman} (see also \cite[Example~2]{beer}).
At a first glance, it may appear that the Baire property of $(X,\norm{\cdot})$ can be used to show that $\bigcup_{n\in\N} E_n =X$ implies that $E_n=X$ for all $n\ge n_0$ for some $n_0\in\N$ as in Example \ref{example}. Unfortunately, this is not the case, as every  linear subspace of $X$ which contains $X_{\mathrm{fin}}$ is necessarily $\norm{\cdot}$-open, and so of second category, by \cite[Corollary~3.9]{beer}. Hence, in this case the Baire property does not yield any useful information.
More generally, the same arguments sketched above can be made also in the case when $X$ is a Hausdorff countable extended seminormed vector space. This follows from the results of \cite[Section~4.1]{salasgarcia} and \cite[Proposition~3.19]{salasgarcia}.

\end{rem}

The problem outlined in the above remark indicates that the Baire-type arguments are on conventional Banach spaces much more powerful. However, as already noted at the end of the previous section, the use  of the classical Baire category theorem has the serious drawback that one needs closedness of the involved subspaces. As an alternative strategy, we use extended seminorms to replace the closedness condition.

\begin{defi}\label{definition P subspace}
Let $V$ be a topological vector space and let $E\subset V$ be a subspace. We say that $E$ is a $P$-\textit{subspace}
of $V$ if the following situation occurs.
There exists an index set $I$ with a countable set $Q_i$ for each $i\in I$ and an extended seminorm $p_{i,q}:V\rightarrow\rplc$ for each $(i,q)\in \bigcup_{i\in I}\big( \{i\}\times Q_i\big)$ such that
\begin{enumerate}[label=(\roman*), font=\normalfont]
\item for all $i\in I$ and all $q\in Q_i$ the ball $[p_{i,q}\le 1]$ is closed in $V$;
\item the following inclusion is satisfied:
\[
E\subset \bigcap_{i \in I}\bigcup_{q\in Q_i} [p_{i,q}<\infty];
\]
\item for every net $(x_\alpha)_{\alpha}$ in $E$ converging to $x\in V$ the following implication holds true:
\[
\left(\forall\, i\in I \ \ \  \exists\, q_i\in Q_i \ \ \  p_{i,q_i}(x-x_\alpha)\sto{\alpha} 0\right) \quad\Rightarrow\quad x\in E.
\]
\end{enumerate}
\end{defi}

\begin{nota}\label{nota without index}
In order to simplify the notation, we will write no index $q_i$ for the set $Q_i$ (or $i$ for the set $I$) under $p$ whenever $Q_i$ (or $I$) is a singleton. If we have a sequence of $P$-subspaces $(E_n)_{n\in\N}$ and we consider for each $E_n$ a specific family of seminorm satisfying the conditions of Definition \ref{definition P subspace}, then we denote by $p_{n;i,q_i}$ the seminorms corresponding to the subspace $E_n$.
\end{nota}

\begin{rem}\label{rem family of seminorms}
Suppose that $E$ is a $P$-subspace such that each set $Q_i$ is a singleton. In this case, assertion (ii) from Definition \ref{definition P subspace} reduces to 
\[
E\subset F\coloneqq \bigcap_{i \in I}\, [p_{i}<\infty].
\]
In the terminology of Salas and Tapia-García, we have that $F=X_{\mathrm{fin}}$ (see \cite[(4) at p.~327 and Proposition~3.10]{salasgarcia}). Since each extended seminorm $p_i$ is a seminorm on $F$, the family $(p_i)_{i\in I}$ induces a locally convex topology $\tau_I$ on $F$. If $E$ is closed in $F$ with respect to the topology $\tau_I$, then assertion (iii) holds true for $E$. Moreover, if $I$ is countable, then $(F,\tau_I)$ is metrizable. Hence, in order to prove that $E$ is $\tau_I$-closed in $F$, it suffices to check that it is $\tau_I$-sequentially closed in $F$.
\end{rem}

\begin{rem}\label{rem condition (iv)}
In the context of Definition \ref{definition P subspace}, consider the following assertion.
\begin{itemize}
\item[(iv)] For every sequence $(x_n)_{n\in\N}$ in $E$ converging to $x\in V$ the following implication holds true:
\[
\left(\forall\, i\in I \ \ \  \exists\, q_i\in Q_i \ \ \  p_{i,q_i}(x-x_n)\sto{n\to \infty} 0\right) \quad\Rightarrow\quad x\in E.
\]
\end{itemize}
It is clear that if $E$ satisfies assertion (iii) from Definition \ref{definition P subspace}, then it satisfies also assertion (iv). On the other hand, assertions (iii) and (iv) are equivalent whenever the topological vector space $V$ satisfies the first axiom of countability. Indeed, suppose that $V$ is first-countable and let $(x_\alpha)_\alpha$ be a net in $E$ converging to $x\in V$. As $V$ is first-countable, $(x_\alpha)_\alpha$ admits a subsequence (i.e. a countable subnet) $(x_{\alpha_n})_{n\in\N}$ also converging to $x$. Now, for each $i\in I$, $\big(p_{i,q_i}(x-x_{\alpha_n})\big)_{n\in\N}$ is a subsequence of the net $\big(p_{i,q_i}(x-x_{\alpha})\big)_{\alpha}$. Hence, if $p_{i,q_i}(x-x_{\alpha})\sto{\alpha} 0$, then also $p_{i,q_i}(x-x_{\alpha_n})\to 0$ as $n\to\infty$. It follows that, if assertion (iv) holds true, then $x\in E$, which means that also assertion (iii) holds true for $E$.
\end{rem}

For the proof of the next two results we need the following lemma.

\begin{lemma}\label{seminorm contin}
Let $V$ be a topological vector space and let $p: V\rightarrow\rplc$ be an extended seminorm such that $\interior\big([p\le 1]\big)\neq\emptyset$. Then $p$ is continuous and $p(V)\subset\rpl$.
\end{lemma}

\begin{proof}
Let $U\subset [p\le 1]$ be open and let $u\in U$. Then $\frac{1}{2}(U-u)$ is contained in $[p\le 1]$ and it is an open neighbourhood of $0$ by \cite[Chapter I, 1.1.(i)]{schaefer}. From \cite[Chapter I, 1.2]{schaefer}, it follows that $p(V)\subset\rpl$ and thus $p$ is continuous in $0$. Therefore, $p$ is also (uniformly) continuous by \cite[Chapter II, 1.6]{schaefer}.
\end{proof}

\begin{prop}\label{prop closed or meagre}
Let $V$ be a topological vector space. If $E\subset V$ is a $P$-subspace, then it is closed or meagre.
\end{prop}

\begin{proof}
We distinguish two cases. In the first case assume that 
\[
\forall\, i\in I \quad \exists\, q_i\in Q_i \qquad \interior\big([p_{i,q_i}\le 1]\big)\neq \emptyset.
\]
Then $V=[p_{i,q_i}<\infty]$ and $p_{i,q_i}$ is continuous for each $i\in I$ by Lemma \ref{seminorm contin}. Thus, $E$ is closed in $V$ by assertion (iii) from Definition \ref{definition P subspace}. \\
In the second case assume that
\[
\exists\, i\in I \quad \forall\, q\in Q_i \qquad \interior\big([p_{i,q}\le 1]\big)= \emptyset.
\]
Now, assertion (ii) from Definition \ref{definition P subspace} implies that 
\[
E\subset \bigcup_{q\in Q_i} [p_{i,q}<\infty]=\bigcup_{q\in Q_i} \bigcup_{k\in\N} [p_{i,q}\le k].
\]
From assertion (i), each ball $[p_{i,q}\le k]$ is nowhere dense and therefore $E$, being a subset of a meagre set, is meagre.
\end{proof}

We come now to a Baire-type theorem version specifically for $P$-subspaces.

\begin{prop}\label{prop baire}
Let $V$ be a topological vector space which is also a Baire space. If $(E_n)_{n\in \N}$ is a sequence of $P$-subspaces of $V$ such that 
\[
\bigcup_{n\in \N} E_n =V,
\]
then there is $n_0\in \N$ with $E_{n_0}=V$. 
\end{prop}

\begin{proof}
Set 
\[
N_1\coloneqq\{n\in \N\st E_n \mbox{ is closed}\} \qquad \mbox{and} \qquad N_2\coloneqq\{n\in \N\st E_n \mbox{ is meagre}\}.
\]
By Proposition \ref{prop closed or meagre}, $\N=N_1\cup N_2$, hence 
\[
V=\bigcup_{n\in N_1} E_n \cup \bigcup_{n\in N_2} E_n.
\]
Since $V$ is a Baire space, it is not meagre in itself. As $V$ is a countable union, there is some $n_0\in \N$ such that $E_{n_0}$ is a subspace of second category in $V$. Necessarily, $n_0\in N_1$, which means that $E_{n_0}$ is even closed. Finally, $E_{n_0}=V$ as every subspace of second category is dense in $V$.
\end{proof}

From Proposition \ref{prop baire} we deduce the next result, which is the basis for most of the proofs in Section \ref{section applications}.

\begin{corolla}\label{corolla}
Let $u:\rpl\rightarrow \caL(X;Y)$ be a function and let $\mathtt{P}$ be a regularity property like in Definition \ref{definition of P}. Assume that, for each $n\in\N$, the set
\[
E_n\coloneqq \big\{x\in X\st u_x \mbox{ has } \mathtt{P} \mbox{ on } \R_{> n} \big\}
\]
is a $P$-subspace of $X$. Then, $u$ has individually eventually the property $\mathtt{P}$ if and only if it has $\mathtt{P}$ uniformly eventually.
\end{corolla}


\section{Applications}\label{section applications}

We provide now many different applications of Corollary \ref{corolla}. In particular, we show that the notions ``uniformly eventually'' and ``individually eventually'' coincide for several linear properties. Throughout this section we will freely use the following observation.

\begin{rem}\label{remark supremum continuous}
Assume that $X$ is a topological space and, for some index set $I$, let $(f_i)_{i\in I}$ be a family of continuous functions $f_i: X\rightarrow \R$. Then the pointwise supremum 
\[
f: X\rightarrow\R\cup\{+\infty\}, \qquad f(x)\coloneqq \sup_{i\in I} f_i(x)
\]
is lower semicontinuous. In particular, $[f\le t]$ is closed in $X$ for all $t\in\R\cup\{+\infty\}$.
\end{rem}

\subsection{Boundedness}
We start by considering boundedness and local boundedness.
\begin{prop}
A function $u:\rpl\rightarrow\caL(X;Y)$ is individually eventually bounded if and only if it is uniformly eventually bounded.
\end{prop}

\begin{proof}
By Corollary \ref{corolla}, it suffices to prove that
\[
E_n\coloneqq \big\{x\in X \st u_x \mbox{ is bounded on } \R_{> n}\big\}
\]
is a $\mathtt{P}$-subspace of $X$ for each $n\in\N$. With each subspace $E_n$, we associate the extended seminorm
\[
p_n:X\rightarrow\rplc, \qquad p_n(x)\coloneqq\sup\big\{\norm{u_x(t)}\st t\in\R_{> n}\big\},
\]
cf.\ Notation \ref{nota without index}.
Due to Remark \ref{remark supremum continuous}, $[p_n\le 1]$ is closed in $X$, which shows that assertion (i) of Definition \ref{definition P subspace} is satisfied. Finally, assertions (ii) and (iii) are trivially satisfied because $E_n=[p_n<\infty]$, and this concludes the proof.
\end{proof}

\begin{prop}\label{ev bdd}
A function $u:\rpl\rightarrow\caL(X;Y)$ is individually eventually locally bounded if and only if it is uniformly eventually locally bounded.
\end{prop}

\begin{proof}
The claim follows from Corollary \ref{corolla}. We need to show that the subspace
\[
E_n\coloneqq \big\{x\in X \st u \mbox{ is locally bounded on } \R_{> n} \big\}
\]
is a $P$-subspace for each $n\in\N$. For every $n,m\in\N$ consider the extended seminorm
\[
p_{n;m}:X\rightarrow\rplc, \qquad p_{n;m}(x)\coloneqq  \sup\bigg\{\big\lVert u_x(t)\big\rVert\st t\in \Big[n+\frac{1}{m},n+m\Big]\bigg\}
\]
and observe that $[p_{n;m}\le 1]$ is closed in $X$ by Remark \ref{remark supremum continuous}.
As for any $n\in\N$
\[
E_n=\bigcap_{m\in\N} [p_{n;m}<\infty],
\]
assertions (ii) and (iii) of Definition \ref{definition P subspace} are trivially satisfied. Thus, $E_n$ is a $P$-subspace.
\end{proof}

\subsection{Continuity}
Firstly, we turn our attention to Hölder continuity.
\begin{prop}
Let $0 <\alpha\le 1$. A function $u:\rpl\rightarrow\caL(X;Y)$ is individually eventually $\alpha$-Hölder continuous if and only if it is uniformly eventually $\alpha$-Hölder continuous.
\end{prop}

\begin{proof}
For each $n\in\N$ consider the subspace
\[
E_n=\{x\in X \st u_x \mbox{ is }  \alpha\mbox{-Hölder continuous on } \R_{> n}\}
\]
and the extended seminorm 
\[
p_n:X\rightarrow\rplc, \qquad p_n(x)\coloneqq \sup_{n\le r<s}\frac{\norm{u_x(s)-u_x(r)}}{(s-r)^\alpha}.
\]
By Remark \ref{remark supremum continuous}, the ball $[p_n\le 1]$ is closed in $X$, thus $p_n$ satisfies assertion (i) from Definition \ref{definition P subspace}. Moreover, since $E_n=[p_n<\infty]$, assertions (ii) and (iii) hold trivially and therefore $E_n$ is a $P$-subspace. Now the claim follows from Corollary \ref{corolla}.
\end{proof}

\begin{prop}
Let $0<\alpha\le 1$. A function $u:\rpl\rightarrow\caL(X;Y)$ is individually eventually locally $\alpha$-Hölder continuous if and only if it is uniformly eventually locally $\alpha$-Hölder continuous.
\end{prop}

\begin{proof}
For each $n\in\N$ consider the subspaces
\[
E_n\coloneqq \big\{x\in X\st u_x \mbox{ is locally } \alpha\mbox{-Hölder continuous on } \R_{> n}\big\}.
\]
Once proven that each $E_n$ is a $P$-subspace, then the claim follows from Corollary \ref{corolla}. For every $n\in\N$ and $m\in \N\setminus\{1\}$ define the extended seminorm
\[
p_{n;m}:X\rightarrow\rplc, \qquad p_{n;m}(x)\coloneqq \sup\bigg\{\frac{\norm{u_x(s)-u_x(r)}}{(s-r)^\alpha} \st n+\frac{1}{m} \le r<s\le n+m\bigg\}
\]
and observe that $[p_{n;m}\le 1]$ is closed in $X$ by Remark \ref{remark supremum continuous}. Thus, $p_{n;m}$ satisfies assertion (i) from Definition \ref{definition P subspace}, while $E_n$ satisfies assertions (ii) and (iii) as
\[
E_n = \bigcap_{m=2}^\infty [p_{n;m}<\infty],
\]
which shows that $E_n$ is a $P$-subspace.
\end{proof}


\begin{prop}\label{cont}
A function $u:\rpl\rightarrow\caL(X;Y)$ is individually eventually continuous if and only if it is uniformly eventually continuous.
\end{prop}

\begin{proof}
We show that 
\[
E_n\coloneqq \big\{x\in X\st u_x \mbox{ is continuous on } \R_{> n}\big\}
\]
is a $P$-subspace of $X$ for each $n\in\N$, so that the claim follows from Corollary \ref{corolla}. \\
For each $n,m\in\N$ consider the extended seminorm
\[
p_{n;m} :X\rightarrow \rplc, \qquad p_{n;m}(x)\coloneqq \sup\bigg\{\norm{u_x(t)-u_x(s)}\st s,t\in \Big[n+\frac{1}{m},n+m\Big] \bigg\},
\]
and observe that by Remark \ref{remark supremum continuous}, $[p_{n;m}\le 1]$ is closed in $X$, hence $p_{n;m}$ satisfies assertion (i) of Definition \ref{definition P subspace}. Moreover,
\[
E_n\subset F_n\coloneqq \bigcap_{m\in\N} [p_{n;m}<\infty],
\]
which implies assertion (ii).
In order to check that $E_n$ is closed in $F_n$ with respect to the topology induced by the family $P_n\coloneqq (p_{n;m})_{m\in\N}$ on $F_n$, (cf. with Remark \ref{rem family of seminorms}), take a sequence $(x_j)_{j\in\N}$ in $E_n$ converging to some $x\in F_n$ with respect to the topology induced by $P_n$, that is, $p_{n;m}(x-x_j)\to 0$ as $j\to\infty$ for every $m\in\N$. This implies that $u_{x_j}$ converges to $u_x$ on $\R_{> n}$ locally uniformly. Hence, $u_x$ is continuous on $\R_{>n}$, and therefore $x\in E_n$. Thus, also assertion (iii) is satisfied and $E_n$ is a $P$-subspace.
\end{proof}

The assertion of the following proposition contains an unwanted but necessary wordplay.

\begin{prop}
A function $u:\rpl\rightarrow\caL(X;Y)$ is individually eventually uniformly continuous if and only if it is uniformly eventually uniformly continuous.
\end{prop}

\begin{proof}
For each $n\in\N$ define the subspace
\[
E_n\coloneqq \big\{x\in X\st u_x\mbox{ is uniformly continuous on } \R_{>n}\big\}.
\]
The claim follows from Corollary \ref{corolla} once shown that each $E_n$ is a $P$-subspace. \\
To this end, define for each $n\in\N$ the extended seminorm
\[
p_n:X\rightarrow\rplc, \qquad p_n(x)\coloneqq \sup\big\{\norm{u_x(t)-u_x(s)}\st s,t \in\R_{> n}, \,\,\, \abs{t-s}\le 1\big\}.
\]
Thanks to the uniform continuity, $E_n\subset[p_n<\infty]$ and by Remark \ref{remark supremum continuous} $[p_n\leq 1]$ is closed in $X$, which means that assertions (i) and (ii) from Definition \ref{definition P subspace} are satisfied.
Taking into account Remark \ref{rem family of seminorms}, it suffices to show that $E_n$ is closed with respect to the topology induced in $F_n\coloneqq [p_n<\infty]$ by the seminorm $p_n$. So, take now a sequence $(x_j)_{j\in\N}$ in $E_n$ and let $x\in F_n$ such that $p_n(x-x_j)\to 0$ as $j\to\infty$. For $j\in\N$ and $s,t > n$ with $\abs{t-s}\le 1$ one obtains
\begin{align*}
\nnorm{u_x(t)-u_x(s)} &\leq \nnorm{u_{x-x_j}(t)-u_{x-x_j}(s)}+\nnorm{u_{x_j}(t)-u_{x_j}(s)} \\
&\leq p_n(x-x_j)+\nnorm{u_{x_j}(t)-u_{x_j}(s)},
\end{align*}
which, combined with the uniform continuity of $u_{x_j}$ on $\R_{>n}$, yields that
\[
\lim_{\delta\to 0}\sup\big\{\norm{u_x(t)-u_x(s)}\st s,t\in\R_{>n} \mbox{ and } \abs{t-s}<\delta\big\} \le p_n(x-x_j).
\]
As $j$ is chosen arbitrarily, the term on the left-side is equal to $0$ and this shows that $x\in E_n$. Thus, $E_n$ is a $P$-subspace.
\end{proof}

Let $I\subset\R$ be an interval. We recall that the \textit{total variation} of a function $f:I\rightarrow X$ on a compact interval $K=[a,b]\subset I$ is the quantity
\[
V_K(f)\coloneqq V_a^b(f)\coloneqq\sup\bigg\{\sum_{j=1}^n \nnorm{f(t_j)-f(t_{j-1})} \st t_0=a<t_1<\ldots<t_n=b\bigg\}.
\]
The function $f$ is said to be \textit{locally of bounded variation} if $V_K(f)<\infty$ for every compact interval $K\subset I$. It is well known that
\begin{align}\label{equation total variation}
V_a^c(f)=V_a^b(f)+V_b^c(f) \qquad\qquad \mbox{ for all } a,b,c\in I \mbox{ with } a<b<c
\end{align}
and for every other function $g:I\rightarrow X$ and $\lambda\in\C$
\[
V_K(\lambda f)=\abs{\lambda}V_K(f), \qquad V_K(f+g)\le V_K(f)+V_K(g).
\]
From this second relation, it turns out that $V_K$ has the seminorm property for every compact interval $K\subset I$. We have then the following proposition.

\begin{prop}
A function $u:\rpl\rightarrow\caL(X;Y)$ is individually eventually locally of bounded variation if and only if it is uniformly eventually locally of bounded variation.
\end{prop}

\begin{proof}
For $n\in\N$ define the subspaces 
\[
E_n\coloneqq \big\{x\in X\st u_x \mbox{ is locally of bounded variation on } \R_{>n}\big\}.
\]
If each vector space $E_n$ is a $P$-subspace of $X$, then the claim follows from Corollary \ref{corolla}. Now, for each $n,m\in\N$ define the compact interval $K_{n;m}\coloneqq \big[n+\frac{1}{m},n+m\big]$ and the extended seminorm
\[
p_{n;m}:X\rightarrow\rplc, \qquad p_{n;m}(x)\coloneqq V_{K_{n;m}}(u_x).
\]
By Remark \ref{remark supremum continuous}, assertion (i) from Definition \ref{definition P subspace} is satisfied, while each subspace $E_n$ satisfies assertions (ii) and (iii) since
\[
E_n=\bigcap_{m\in\N} [p_{n;m}<\infty].
\]
Thus, $E_n$ is a $P$-subspace for each $n\in\N$.
\end{proof}

Using a similar method, we obtain also the following result.

\begin{prop}
A function $u:\rpl\rightarrow\caL(X;Y)$ is individually eventually locally absolutely continuous if and only if it is uniformly eventually locally absolutely continuous.
\end{prop}

\begin{proof}
The claim follows from Corollary \ref{corolla} once proven that for each $n\in\N$
\[
E_n\coloneqq \big\{x\in X\st u_x \mbox{ is locally absolutely continuous on } \R_{>n}\big\}
\]
is a $P$-subspace of $X$. Define the compact interval $K_{n;m}\coloneqq \big[n+\frac{1}{m},n+m\big]$ for each $n,m\in\N$ together with the extended seminorm
\[
p_{n;m}:X\rightarrow\rplc, \qquad p_{n;m}(x)\coloneqq V_{K_{n;m}}(u_x).
\]
By Remark \ref{remark supremum continuous}, assertion (i) from Definition \ref{definition P subspace} is satisfied. Since every function which is absolutely continuous on a compact interval is of bounded variation on the same interval, we have 
\[
E_n\subset F_n\coloneqq \bigcap_{m\in\N} [p_{n;m}<\infty],
\]
which shows that $E_n$ satisfies assertion (ii).
Concerning assertion (iii), consider a sequence $(x_j)_{j\in\N}$ in $E_n$ and $x\in F_n$ such that $p_{n;m}(x-x_j)\to 0$ as $j\to\infty$ for all $m\in\N$ (cf. Remark \ref{rem family of seminorms}). Fix $m\in\N$ and let $(a_k)_{k=1}^l$ and $(b_k)_{k=1}^l$ be finite sequences in $K_{n;m}$ with $a_k<b_k$ for each $k\in\{1,\ldots,l\}$. By \eqref{equation total variation}, the inequality
\begin{align*}
\sum_{k=1}^l \nnorm{u_x(b_k)-u_x(a_k)}& \le \sum_{k=1}^l \nnorm{(u_{x-x_j})(b_k)-(u_{x-x_j})(a_k)}+\sum_{k=1}^l\nnorm{u_{x_j}(b_k)-u_{x_j}(a_k)}\\
&\le p_{n;m}(x-x_j)+\sum_{k=1}^l\nnorm{u_{x_j}(b_k)-u_{x_j}(a_k)}
\end{align*}
holds true for each $j\in\N$. Consider for each $\delta>0$ the quantity 
\[
D_\delta\coloneqq \sup \bigg\{\sum_{k=1}^l\nnorm{u_x(b_k)-u_x(a_k)}\st  \sum_{k=1}^l (b_k-a_k)<\delta \bigg\},
\]
where the supremum is taken over $l\in\N$ and all finite sequences $(a_k)_{k=1}^l$, $(b_k)_{k=1}^l$, in $K_{n;m}$ like above.
From the local absolute continuity of $u_{x_j}$ on $\R_{>n}$ it follows that
\[
\lim_{\delta\to 0} D_\delta \le p_{n;m}(x-x_j)
\]
Since $j\in\N$ is chosen arbitrarily, the term on the left-side is equal to $0$. This shows that $u_x$ is locally absolutely continuous on $\R_{>n}$ and, so, that $x\in E_n$. Thus, $E_n$ is a $P$-subspace.
\end{proof}

\begin{prop}
A function $u:\rpl\rightarrow\caL(X;Y)$ is individually eventually absolutely continuous if and only if it is uniformly eventually absolutely continuous.
\end{prop}

\begin{proof}
The claim follows from Corollary \ref{corolla} once proven that for each $n\in\N$
\[
E_n\coloneqq \big\{x\in X\st u_x \mbox{ is absolutely continuous on } \R_{>n}\big\}
\]
is a $P$-subspace of $X$. For each $n\in\N$ and $\delta>0$ define
\[
S_n(x,\delta)\coloneqq \bigg\{\sum_{k=1}^l \norm{u_x(b_k)-u_x(a_k)}\st l\in\N,\,\, n<a_1<b_1<\ldots<a_l<b_l,\,\,\, \sum_{k=1}^l (b_k-a_k) < \delta\bigg\} 
\]
together with the extended seminorms
\[
p_n:X\rightarrow\rplc, \qquad p_n(x)\coloneqq \sup S_n(x,1).
\]
Fix $n\in\N$. Of course, $[p_n\le 1]$ is closed in $X$ by Remark \ref{remark supremum continuous} and $E_n\subset F_n\coloneqq [p_n<\infty]$. This shows that assertions (i) and (ii) from Definition \ref{definition P subspace} are satisfied. It remains to show that $E_n$ satisfies assertion (iii), that is, that $E_n$ is closed in $F_n$ with respect to the topology induced by $p_n$ on $F_n$. Taking into account Remark \ref{rem family of seminorms}, let $(x_j)_{j\in\N}$ be a sequence in $E_n$ and let $x\in F_n$ such that $p_n(x-x_j)\to 0$ as $j\to \infty$. For $l\in\N$ and disjoint intervals $[a_1,b_1],\ldots, [a_l,b_l]\subset\R_{>n}$ with $\sum_{k=1}^l (b_k-a_k)\le 1$ we have
\begin{align*}
\sum_{k=1}^l \nnorm{u_x(b_k)-u_x(a_k)}& \le \sum_{k=1}^l \nnorm{(u_x-u_{x_j})(b_k)-(u_x-u_{x_j})(a_k)}+\sum_{k=1}^l\nnorm{u_{x_j}(b_k)-u_{x_j}(a_k)}\\
&\le p_{n}(x-x_j)+\sum_{k=1}^l\nnorm{u_{x_j}(b_k)-u_{x_j}(a_k)}.
\end{align*}
At this point, using the absolute continuity of $u_{x_j}$ on $\R_{>n}$ and passing to the limit in the above expression, we obtain
\[
\lim_{\delta\to 0} \sup S_n(x,\delta)\le p_n(x-x_j)
\]
Since $j\in\N$ is chosen arbitrarily, $\lim_{\delta\to 0} \sup S_n(x,\delta)=0$, which means that $u_x$ is absolute continuous on $\R_{>n}$ and so $x\in E_n$. Thus, $E_n$ satisfies assertion (iii) from Definition \ref{definition P subspace} and is thus a $P$-subspace.
\end{proof}

\subsection{Differentiability}\label{section differentiability}

As mentioned in the introduction, Proposition \ref{differentiab} generalizes the results on eventual differentiability obtained by Iley and Bárta in \cite{iley} and \cite{barta}, respectively, by dispensing with the assumption of strong continuity for $u$.

\begin{prop}\label{differentiab}
A function $u:\rpl\rightarrow\caL(X;Y)$ is individually eventually differentiable if and only if it is uniformly eventually differentiable.
\end{prop}

In order to prove the proposition and also for the next applications it is convenient \linebreak to recall the following classical fact, which can be found, for example, in \linebreak \cite[Proposition~3.3.3]{tao}.

\begin{lemma}[Interchange of Limits]\label{lemma interchange of limits}
Let $(X,d_X)$ and $(Y,d_Y)$ be metric spaces with $Y$ complete, and let $S$ be a subset of $X$. Let $(f_j)_{j\in\N}$ be a sequence of functions from $S$ to $Y$ and suppose that it converges uniformly on $S$ to some function $f:S\rightarrow Y$. Let $s_0\in \overline{S}$ and suppose that for each $j\in\N$ the limit $\lim_{s\to s_0, \, s\in S} f_n(s)$ exists. Then the limit $\lim_{s\to s_0, \, s\in S} f(s)$ also exists and it is possible interchange the limits, that is,
\[
\lim_{n\to\infty} \,\, \lim_{\substack{s\in S \\ s\to s_0}} f_n(s)\,\,\,\,\,\, = \,\,\,\,\lim_{\substack{s\in S \\ s\to s_0}} f(s).
\]
\end{lemma}

\begin{proof}[Proof of Proposition \ref{differentiab}]
It suffices to apply Corollary \ref{corolla}. For this, we need to show that, for each $n\in\N$, the subspace
\begin{align}\label{spaces differentiabiliy}
E_n\coloneqq \big\{x\in X\st u_x \mbox{ is differentiable on } \R_{>n} \big\}.
\end{align}
is a $P$-subspace of $X$. \\
For each $n\in\N$ and $t\in\R_{> n}$ choose $\delta_{n;t}\in\R_{>0}$ such that $[t-\delta_{n;t},t+\delta_{n;t}]\subset\R_{>n}$. Then, put $I_{\delta_{n;t}}\coloneqq [-\delta_{n;t},\delta_{n;t}]\setminus\{0\}\subset\rpl$ for the sake of simplicity and define the extended seminorm
\[
p_{n;t} :X\rightarrow \rplc,\qquad\qquad p_{n;t}(x)\coloneqq \sup_{h\in I_{\delta_{n;t}}}\,\,\,\,\,\normm{\frac{u_x(t+h)-u_x(t)}{h}}.
\]
By Remark \ref{remark supremum continuous}, $[p_{n;t}\le 1]$ is closed in $X$ and each extended seminorm $p_{n;t}$ satisfies assertion (i) from Definition \ref{definition P subspace}.
It follows that 
\[
E_n\subset F_n\coloneqq \bigcap_{t>n} [p_{n;t}<\infty].
\]
Thus, $E_n$ satisfies assertion (ii) with respect to the family of seminorms $P_n\coloneqq (p_{n;t})_{t>n}$. It remains to show that assertion (iii) holds true for $E_n$ and the family $P_n$. By Remark \ref{rem condition (iv)}, it suffices to consider assertion (iv) instead of (iii). Let $(x_j)_{j\in\N}$ be a sequence in $E_n$ and let $x\in F_n$ such that $p_{n;t}(x-x_j)\to 0$ as $j\to\infty$ for each $t\in\R_{>n}$.
Fix $t\in\R_{>n}$, and let $f_j : I_{\delta_{n;t}}\rightarrow Y$, $j\in\N$, and $f : I_{\delta_{n;t}}\rightarrow Y$ given by
\[
f_j(h)\coloneqq \frac{u_{x_j}(t+h)-u_{x_j}(t)}{h}\,\,\, (j\in\N) \quad \mbox{and} \quad  f(t)\coloneqq \frac{u_x(t+h)-u_x(t)}{h}.
\]
Since $p_{n;t}(x-x_j)\to 0$ as $j\to\infty$, the sequence $(f_j)_{j\in\N}$ converges uniformly to $f$ on $I_{\delta_{n;t}}$. Moreover, since $x_j\in E_n$ for each $j\in\N$, the limit $\lim_{h\to 0, \, h\in I_{\delta_{n;t}}} f_j(h)$ exists for each $j\in \N$. Now, Lemma \ref{lemma interchange of limits} applies with $S=I_{\delta_{n;t}}$ and it guarantees the existence of the limit $\lim_{h\to 0, \, h\in I_{\delta_{n;t}}} f(t)$, which implies the differentiability of $u_x$ at $t$. As $t$ is chosen arbitrarily, we deduce that $x\in E_n$ and assertion (iv) is therefore satisfied. Thus, $E_n$ is a $P$-subspace of $X$.
\end{proof}


\begin{lemma}\label{deriv lin op}
Let $k\in\N$ and let $u:\rpl\rightarrow\caL(X;Y)$ be a function. If $I\subset\R_{>0}$ is an open interval such that $u_x$ is continuously $(k-1)$-times differentiable on $I$ and $k$-times differentiable on some $t\in I$ for all $x\in X$, then
\begin{align}\label{equation strong derivative}
u^{(k)}(t):X\rightarrow Y,\qquad u^{(k)}(t)x\coloneqq u_x^{(k)}(t)
\end{align}
is an operator in $\caL(X;Y)$.
\end{lemma}
\begin{proof}
We prove the claim by induction on $k$.
Since the proof of the claim in the base case when $k=1$ is very the much same as that of the step case, we prove directly the induction step. \\
Assume that $u_x$ is $(k-1)$-times differentiable on $I$ and $k$-times differentiable on a fixed $t\in I$. Take $\delta>0$ such that $\{t\}+I_\delta\subset I$, where $I_\delta\coloneqq [-\delta,\delta]\setminus\{0\}$ (cf.\ also Notation \ref{notation Minkowski sum}).
By the induction hypothesis, $u^{(k-1)}(t+h)\in\caL(X;Y)$ for all $h\in I_\delta$, where $u^{(k-1)}(t+h)$ is defined similarly as in \eqref{equation strong derivative}. Now, for each $h\in I_\delta$ consider the operator
\[
S_h\coloneqq \frac{u^{(k-1)}(t+h)-u^{(k-1)}(t)}{h}\in\caL(X;Y).
\]
Since $u_x$ is $k$-differentiable at $t$ for all $x\in X$, it follows that 
\[
\sup_{h\in I_\delta} \Norm{\frac{u_x^{(k-1)}(t+h)-u_x^{(k-1)}(t)}{h}}<\infty\qquad\mbox{for all } x\in X,
\]
thus also $C_\delta\coloneqq \sup_{h\in I_\delta} \norm{S_h}<\infty$ by the uniform boundedness principle. Observe now that, for every null sequence $(h_j)_{j\in\N}$ in $I_\delta$, the sequence $(S_{h_j})_{j\in\N}$ converges strongly to the operator $u^{(k)}(t)$. Hence, $u^{(k)}(t)$ is in $\caL(X;Y)$ with $\norm{u^{(k)}(t)}\leq C_\delta$. 
\end{proof}

\begin{rem}\label{remark strong vs norm derivative}
All derivatives of the function $u$ considered in Lemma \ref{deriv lin op} exist with respect to the strong operator topology. As intuition suggests, if $u$ is $k$-times differentiable at $t$ for some $k\in\N_0$ and $t\in\rpl$ with respect to the strong operator topology, this does not mean in general that $T$ is $k$-times differentiable at $t$ with respect to the operator norm topology. Actually, the question whether ``strong differentiability'' and ``operator norm differentiability'' coincide or not in the case when $u$ is a $C_0$-semigroup turned out to be non-trivial. It has been answered negatively by Bárta: in \cite{barta}, he constructs an eventually differentiable $C_0$-semigroup $T$ on a closed subspace of the Banach space of the bounded uniformly continuous functions on $\rpl$ having the property that $T$ is differentiable with respect to the strong operator topology on the interval $\R_{>2\pi}$, but not differentiable with respect to the operator norm topology at any point of the interval $(2\pi,4\pi)$. In particular, he shows that the differentiability interval in assertion (iv) from Proposition \ref{prop pazy} is optimal.
\end{rem}

Combining Proposition \ref{differentiab} and Lemma \ref{deriv lin op} we easily obtain the following corollary.

\begin{corolla}\label{corolla different}
A function $u:\rpl\rightarrow\caL(X;Y)$ is individually eventually $k$-times differentiable if and only if it is uniformly eventually $k$-times differentiable. 
\end{corolla}


If a function $u:\rpl\rightarrow\caL(X;Y)$ is individually eventually $k$-differentiable for all $k\in\N_0$, then, by Corollary \ref{corolla different}, there is an increasing sequence $(t_k)_{k\in\N}$ of times such that $u$ is also strongly $k$-differentiable on the interval $(t_k,\infty)$ for each $k\in\N_0$. In general, we cannot expect to be able to control the growth of the times $t_k$ i.e., the sequence $(t_k)_{k\in\N}$ cannot be chosen to be bounded, as the following simple example shows.

\begin{es}\label{example t_k unbounded}
Let $\uC_0(\rpl)$ be the space of continuous functions vanishing at infinity and let $a:\rpl \rightarrow\C$ be defined by $a(s)\coloneqq -\log(1+s)+\ui s$. The multiplication semigroup on $\uC_0(\rpl)$ associated to the function $a$, defined by
\[
T:\rpl\rightarrow\caL\big(\uC_0(\rpl)\big), \quad T(t)f(s)\coloneqq f(s)\cdot\ue^{ta(s)} \qquad (f\in \uC_0(\rpl)),
\]
is strongly continuous with generator $A$ given by
\[
\dom(A)=\big\{f\in \uC_0(\rpl)\st a\cdot f\in \uC_0(\rpl) \big\}, \quad Af=a\cdot f.
\]
Fix $n\in\N$. Since 
\begin{align}\label{example asymptotic}
a(s)^n e^{ta(s)}=O\big(s^{n-t}\big) \qquad \mbox{as } s\to\infty \qquad (t\in\rpl),
\end{align}
the operator $A^nT(t)$ is bounded if and only if $t\in\R_{\ge n}$. \\
By Proposition \ref{prop pazy}, $T$ is (uniformly) eventually differentiable for $t_0 = 1$. As $T$ is a $C_0$-semigroup, if the (strong) derivative $T^{(n)}$ exists on an interval of the form $\R_{\ge t_n}$ for some $t_n\in\rpl$, then $T^{(n)}(t)=A^nT(t)\in\caL\big(\uC_0(\rpl)\big)$ for all $t\in\R_{\ge t_n}$. Thus, by \eqref{example asymptotic} the $n$-times derivative $T^{(n)}(t)$ exists if and only if $t\in\R_{\ge n}$.
\end{es}
Example \ref{example t_k unbounded} shows that uniform eventual $k$-differentiability for every $k\in\N_0$ does not imply uniform eventual smoothness in general. Nevertheless, one can wonder if an individually eventually smooth operator-valued function $u$ is also uniformly eventually smooth. It turns out that also to this question the answer is positive.

%
%
%
%
%
%
%
%
%
%

\begin{prop}\label{propsmooth}
A function $u:\rpl\rightarrow\caL(X;Y)$ is individually eventually smooth if and only if it is uniformly eventually smooth.
\end{prop}

We provide two different proofs of Proposition \ref{propsmooth}. The first one makes use of Friedrich's mollifiers and is more direct than the second one. However, even if the second proof is more technical, this last quantitative version explains how to control all derivatives suitably in each point of $\rpl$. \\

We recall that a sequence of Friedrich's mollifiers $(\varrho_l)_{l\in\N}$ defined on $\R$ consists of functions belonging to $\uC^\infty(\R;\R)$ such that, for each $l\in\N$
\begin{align*}
\varrho_l(x)\ge 0 \mbox{ for all } x\in\R, \qquad \supp(\varrho_l)\subset \Big[-\frac{1}{l},\frac{1}{l}\Big], \qquad \int_{\big[-\frac{1}{l},\frac{1}{l}\big]} \varrho_l(x) \de x=1.
\end{align*}
Friedrich's mollifiers have many interesting regularisation properties. In particular, if $f:\R\rightarrow X$ is a function in $\uL_{loc}^1(\R;X)$, then the convolution 
\[
f\ast \rho_l:\R\rightarrow X, \qquad f\ast \rho_l(t) \coloneqq  \int_{\R} f(s)\varrho_l(t-s) \de s
\]
belongs to $\uC^\infty(\R; X)$ and $(f\ast \rho_l)^{(k)}=f\ast \rho_l^{(k)}=f^{(k)}\ast \rho_l$ for each $k,l\in\N$, where the last inequality should be interpreted in the sense of the distributions. Moreover, if $S\subset \R$ is a set on which $f$ is bounded and uniformly continuous, then the sequence $(f\ast\varrho_l)_{l\in\N}$ converges as $l\to\infty$ to the function $f$ uniformly on $S$. Therefore, we obtain the following characterization.

\begin{lemma}\label{lemma characterization of smoothness}
Let $k_0\in\N$. A function $f\in\uC(\R;X)$ belongs to $\uC^{k_0}(\R;X)$ if and only if there exist functions $g_1,\ldots,g_{k_0}\in \uC(\R;X)$ such that, for each $1\le k\le k_0$, the sequence $(f\ast \varrho_l^{(k)})_{l\in\N}$ converges to $g_k$ uniformly on compact subsets. In this case, $g_k=f^{(k)}$ for every $1\le k\le k_0$.
\end{lemma}

\begin{proof}
If $f\in \uC^{k_0}(\R;X)$, then $f^{(k)}\in\uC(\R;X)$ for each $1\le k\le k_0$. In particular, $f^{(k)}$ is uniformly continuous on each compact subset of $\R$ and this implies that $(f\ast \varrho_l^{(k)})_{l\in\N}$ converges uniformly to $f^{(k)}$ on each compact set by the properties of Friedrich's mollifiers.
Conversely, let $K\subset\R$ be compact and assume that continuous functions $g_1,\ldots,g_k$ exist and satisfy the assertion of the lemma. After observing that $f\ast \rho_l^{(k)}=(f\ast \rho_l)^{(k)}$, this simply means that the sequence $\big((f\ast \rho_l)_{\vert K}\big)_{l\in\N}$ is a Cauchy sequence in the Banach space $\big(\uC^k(K;X), \norm{\cdot}_k\big)$, where $\norm{\phi}_k\coloneqq \norm{\phi}_\infty+\norm{\phi'}_\infty+\ldots+\norm{\phi^{(k)}}_\infty$ for $\phi\in \uC^k(K;X)$.
At this point, the claim follows immediately by the arbitrariness of the choice of the compact set $K$.
\end{proof}

\begin{proof}[First proof of Proposition \ref{propsmooth}]
For each $n\in\N$ consider the subspaces
\[
E_n\coloneqq \big\{x\in X \st u_x \mbox{ is smooth on } \R_{>n} \big\}.
\]
In order to apply Corollary \ref{corolla}, we have to prove that each $E_n$ is a $P$-subspace. Take a sequence $(\varrho_l)_{l\in\N}$ of Friedrich's mollifiers and, for each $n,m\in\N$ and $k\in\N_0$, define the compact interval $K_{n; m}\coloneqq \big[n+\frac{1}{m},n+m\big]$ and
\[
p_{n;(m,k)}: X\rightarrow \rplc, \qquad p_{n;(m,k)}(x)\coloneqq \sup\Big\{\nnorm{u_x\ast \varrho_l^{(k)}(t)}\st t\in K_{n;m}, \,\, l\in\N_{\ge 2} \Big\},
\]
where we still write $u$ also for its extension by zero outside $\rpl$ to $\R$. Then, each $p_{n;(m,k)}$ is an extended seminorm. For each $n\in\N$ define the family $P_n\coloneqq (p_{n;(m,k)})_{(m,k)\in\N\times\N}$. \\
As $u$ is individually eventually smooth, by Proposition \ref{cont} there is a starting time $t_0\in\rpl$ such that $u$ is strongly continuous on $\R_{>t_0}$. By passing to the shifted function $\rpl\rightarrow\caL(X;Y)$, $t\mapsto u(t_0+t)$ if necessary, we can assume without loss of generality $t_0=0$. Hence, $u$ turns out to be locally bounded on $\R_{>0}$ by the uniform bounded principle. Therefore, since for $x\in X$, $t\in\R_{\ge 1}$, $l\in\N_{\ge 2}$ and $k\in\N_0$
\[
u_x\ast \varrho_l^{(k)}(t)= \int_{\R} u_x(s)\varrho_l^{(k)}(t-s) \de s=\int_{\big[-\frac{1}{l},\frac{1}{l}\big]} u_x(t-s)\varrho_l^{(k)}(s) \de s,
\]
the function 
\[
X\rightarrow Y, \qquad x\mapsto u_x\ast \varrho_l^{(k)}(t)
\]
is continuous by the dominated convergence theorem. Thus, $[p_{n;(m,k)}\le 1]$ is closed in $X$ and $p_{n;(m,k)}$ satisfies assertion (i) from Definition \ref{definition P subspace} by Remark \ref{remark supremum continuous}. \\ 
With regard to assertion (ii), fix $n\in\N$ and let $x\in E_n$. As $u_x$ is smooth on $\R_{>n}$, by Lemma \ref{lemma characterization of smoothness} the sequence $\big(u_x\ast \varrho_l^{(k)}\big)_{l\in\N}$ converges uniformly on every compact subset of $\R_{>n}$ to $u_x^{(k)}$ for each $k\in\N$. 
From this, $p_{n; (m,k)}(x)<\infty$ for all $m,k\in\N$ and therefore
\[
E_n\subset \bigcap_{(m,k)\in\N\times\N} [p_{n;(m,k)}<\infty].
\]
In order to prove that assertion (iv) from Remark \ref{rem condition (iv)} holds true for $E_n$, take a sequence $(x_j)_{j\in\N}$ such that $p_{n;(m;k)}(x-x_j)\to 0$ as $j\to \infty$ for all $m,k\in\N$. This means that, for each fixed $m,k\in\N$, the sequence $(u_{x_j}\ast \varrho_l^{(k)})_{j\in\N}$ converges uniformly to $u_x\ast\varrho_l^{(k)}$ on $K_{n;m}$ for all $l\in\N$ and even uniformly in $l$.
Besides, as each $x_j\in E_n$, the sequence $(u_{x_j}\ast \varrho_l^{(k)})_{l\in\N}$ converges uniformly to $u_{x_j}^{(k)}$ on the same interval. Thus, by a simple $\epsilon$/3 argument, it follows that $\big(u_{x_j}^{k}\big)_{j\in\N}$ is a Cauchy sequence in $\big(\uC(K_{n;m}), \norm{\cdot}_\infty\big)$. At this point, by the uniform convergence of the limits with respect to the parameter $l$, the limit
\[
\lim_{l\to\infty} u_x\ast \varrho_l^{(k)}= \lim_{l\to\infty}\lim_{j\to\infty} u_{x_j}\ast \varrho_l^{(k)} =\lim_{j\to\infty} \lim_{l\to\infty} u_{x_j}\ast \varrho_l^{(k)} =\lim_{j\to\infty} u_{x_j}^{k}
\]
exists in $\uC(K_{n;m})$. By the arbitrariness of the choice of $k$ and $m$ in $\N$ and by Lemma \ref{lemma characterization of smoothness}, it follows that $x\in E_n$. Hence, $E_n$ is a $P$-subspace of $X$ and so Corollary \ref{corolla} applies and yields the claim.
\end{proof}

For what is coming, it is convenient to introduce the following notation.

\begin{nota}\label{notation}
For $k\in\N$ we write $[k]$ for the set $\{1,\ldots,k\}$. Let $h=(h_1,\ldots,h_k)\in \R^k$. For $J\subset[k]$ we set
\[
h_J\coloneqq 
\left\{
\begin{array}{ll}
0 & \mbox{if } J=\emptyset \\
\sum_{j\in J} h_j & \mbox{if } J\neq\emptyset.
\end{array}
\right.
\]
Let $k\in\N$, let $I\subset\R$ be an open interval and $f:I\rightarrow X$ a function. Fix $t\in I$ and take $\delta\in\R_{>0}$ small enough that $[t-k\delta,t+k\delta]\subset I$. For $I_\delta\coloneqq [-\delta,\delta]\setminus\{0\}$ we define the function 
\[
D_{\delta,k, t}[f]: I_\delta^k\rightarrow X,\quad h\mapsto D_{\delta,k, t}[f](h)\coloneqq \frac{1}{h_1\cdot\ldots\cdot h_k} \sum_{J\subset [k]} (-1)^{k-|J|}\, f(t+h_J).
\]
\end{nota}

The second proof of Proposition \ref{propsmooth} is a natural generalization of the methods used for the proof of Proposition \ref{differentiab}. In this last one, the difference quotient plays the role of an extended seminorm, which allows treating the topic of eventual differentiability. As far as we know, in literature the higher-order derivatives of a function are always defined ``recursively'', in the sense that the derivative of order $k\in\N$ at a given point is obtained as limit of the difference quotient for the $(k-1)$-derivative. In our case, this formulation is not useful: indeed, we need to express the derivative of order $k$ only in dependence of the function itself and not of its derivatives of lower order since we cannot know, in general, if these derivatives exist. The expression $D_{\delta,k, t}[f]$ introduced in Notation \ref{notation} represents a sort of difference quotient of order $k$ for the function $f$. Through it, the existence of the derivative of order $k$ can be characterized and a suitable extended seminorm for controlling its behaviour can be defined. Lemma \ref{das lemma} is preparatory and makes these ideas precise.


\begin{lemma}\label{das lemma}
Suppose that $f$ is $(k-1)$-times differentiable on $I$. Then $f$ is $k$-times differentiable at $t$ if and only if the limit
\[
\lim_{h_k\to 0}\ldots\lim_{h_2\to 0}\lim_{h_1\to 0} D_{\delta,k,t}[f](h) 
\]
exists in $X$. In this case, it coincides with its derivative $f^{(k)}(t)$. Moreover, if $f$ is $k$-times continuously differentiable on $I$, then we have for all $h\in I_\delta^k$ the integral representation
\[
D_{\delta,k, t}[f](h)=\int_0^1\cdots\int_0^1 f^{(k)}\bigg(t+\sum_{j=1}^k h_js_j\bigg)\de s_1\cdots\de s_k.
\]
\end{lemma}

\begin{proof}
The proof is done by induction on $k$. We start by proving the first assertion.\\
$k=1$: In this case the function $f$ is continuous on $I$ and 
\[
D_{\delta,1, t}[f](h_1) = \frac{f(t+h_1)-f(t)}{h_1}
\]
is simply the difference quotient, so the claim is trivially true.
\\
$k\to k+1$: Assume that $f$ is $k$-times differentiable on $I$. We observe first that
\begin{align*}
D_{\delta,(k+1), t}[f](h)&=\frac{1}{h_1\cdot\ldots\cdot h_{k+1}} \sum_{J\subset [k+1]} (-1)^{k+1-|J|}\, f(t+h_J)\\
&=\frac{1}{h_{k+1}} \bigg(\frac{1}{h_1\cdots h_k}\sum_{J\subset [k]} (-1)^{k-|J|}\, f(t+h_J+h_{k+1})\\
&\hspace{2cm}-\frac{1}{h_1\cdots h_k}\sum_{J\subset [k]} (-1)^{k-|J|}\, f(t+h_J)\bigg)\\
&= \frac{1}{h_{k+1}}\Big(D_{\delta,k, (t+h_{k+1})}[f]\big(h_1,\ldots,h_k\big)-D_{\delta,k, t}[f]\big(h_1,\ldots,h_k\big)\Big).
\end{align*}
Using the induction hypothesis and taking the successive limits in the expression above the difference quotient let rewrite itself in the following way:
\begin{align}\label{relation derivative and seminorms}
\frac{f^{(k)}(t+h_{k+1})-f^{(k)}(t)}{h_{k+1}} = \lim_{h_k\to 0}\ldots\lim_{h_1\to 0} D_{\delta,(k+1),t}[f]\big(h_1,\ldots,h_k,h_{k+1}\big).
\end{align}
Clearly, $f$ is $(k+1)$-times differentiable at $t$ if and only if the limit 
\[
\lim_{h_{k+1}\to 0}\frac{f^{(k)}(t+h_{k+1})-f^{(k)}(t)}{h_{k+1}}
\]
exists in $X$.
Thus, this limit exists if and only if the limit 
$$ \lim_{h_{k+1}\to 0}\lim_{h_k\to 0}\ldots\lim_{h_1\to 0} D_{\delta,(k+1),t}[f](h) $$
exists in $X$, and this proves the assertion. \\\\
Let us now establish the integral representation, again by induction. \\
$k=1$: In this case $f$ is differentiable on $I$. Hence, if $h_1\in I_\delta$, we have
\[
\int_0^1 f'(t+h_1s_1)\de s_1=\frac{1}{h_1}\int_0^{h_1} f'(t+s_1)\de s_1=\frac{1}{h_1}\big(f(t+h_1)-f(t)\big)=D_{\delta,1,t}[f](h_1).
\]
$k\to k+1$: Assume that $f$ is $(k+1)$-times differentiable on $I$. For $h_1,\ldots,h_{k+1}\in I_\delta$ we compute, as above,
\begin{equation*}
\begin{aligned}[t]
\int_0^1 &\cdots\int_0^1  f^{(k+1)}\bigg(t+\sum_{j=1}^{k+1} h_js_j\bigg)\de s_1\cdots\de s_{k+1} \\
&= \int_0^1\cdots\int_0^1\frac{1}{h_{k+1}}\left(f^{(k)}\bigg(t+\sum_{j=1}^k h_js_j+h_{k+1}\bigg)-f^{(k)}\bigg(t+\sum_{j=1}^k h_js_j\bigg)\right)\de s_1\cdots\de s_k\\
&=\frac{1}{h_{k+1}}\int_0^1\cdots\int_0^1f^{(k)}\bigg(t+\sum_{j=1}^k h_js_j+h_{k+1}\bigg)\de s_1\cdots\de s_k \\
\MoveEqLeft[-15]-\frac{1}{h_{k+1}}\int_0^1\cdots\int_0^1 f^{(k)}\bigg(t+\sum_{j=1}^k h_js_j\bigg)\de s_1\cdots\de s_k \\
&= \frac{1}{h_{k+1}} \Big(D_{\delta,k,(t+h_{k+1})}[f]\big(h_1,\ldots,h_k\big)-D_{\delta,k,t}[f]\big(h_1,\ldots,h_k\big)\Big)\\
&=D_{\delta,(k+1),t}[f](h).
\end{aligned}
\end{equation*}
This completes the proof of the lemma.
\end{proof}

\begin{proof}[Second proof of Proposition \ref{propsmooth}]
Define the subspaces
\[
E_n\coloneqq \big\{x\in X \st u_x \mbox{ is smooth on the interval } \R_{>n} \big\}.
\]
The assertion of the proposition follows from Corollary \ref{corolla} if we show that each $E_n$ is a $P$-subspace of $X$. Fix $n\in\N$. We define a family of seminorms using two indices, representing the derivative or order $k$ and the point $t$ of the interval $\R_{>n}$, i.e.
\[
P_n\coloneqq \big(p_{n;(k,t)}\big)_{(t,k)\in\R_{>n}\times\N}, 
\]
where each seminorm is defined in the following way: for $k\in\N$ and $t\in \R_{>n}$ take $0<\delta_{n;(k,t)}<\frac{t-n}{k}$ and let 
\[
I_{n; (t,k)}\coloneqq I_{\delta_{n;(k,t)}}\coloneqq [-\delta_{n;(k,t)},\delta_{n;(k,t)}]\setminus\{0\}.
\]
For the sake of simplicity, let us write for $x\in X$, $n,k\in\N$ and $t>n$
\[
D_{n;(k,t)}[x]: I_{n; (t,k)}^k\rightarrow Y,\qquad D_{n;(k,t)}[x]\coloneqq D_{\delta_{n;(t,k)}, k, t}[u_x],
\]
where the latter function, as the interval above, has been defined in Notation \ref{notation}. For each $n\in\N$ define the seminorms
\[
p_{n;(k,t)}:X\rightarrow\rpl,\qquad p_{n;(k,t)}(x) \coloneqq \sup\Big\{\Norm{D_{n;(k,t)}[x](h)}\st h\in I_{\delta_{n;(t,k)}}^k\Big\}.
\]
Clearly, $p_{n;(t,k)}$ is an extended seminorm for each $k\in\N$ and $t>n$ and $[p_{n;(t,k)}\le 1]$ is closed in $X$ by Remark \ref{remark supremum continuous}. Thus, $p_{n;(t,k)}$ satisfies assertion (i) from Definition \ref{definition P subspace}.
In order to see that 
\[
E_n\subset \bigcap_{(t,k)\in\R_{>n}\times\N} [p_{n;(k,t)}<\infty],
\]
we use the integral representation in the second part of Lemma \ref{das lemma}. If $x\in E_n$, then $u_x$ is smooth on the interval $\R_{>n}$ and hence on $[t-k\cdot\delta_{n;(t,k)}, t+k\cdot\delta_{n;(t,k)}]$, as it is contained in $\R_{>n}$ by the choice of $\delta_{n; (k,t)}$. Hence, for $t>n$ and $k\in\N$ we have 
\begin{align*}
p_{n;(k,t)}(x)&=\sup\Bigg\{\normm{\int_0^1\cdots\int_0^1 u_x^{(k)}\bigg(t+\sum_{j=1}^k h_js_j\bigg)\de s_1\cdots\de s_k} \st h\in I_{\delta_{n;(t,k)}}^k\Bigg\} \\
&\le \norm{u_x^{(k)}}_{\infty, [t-k\cdot\delta_{n;(t,k)}, t+k\cdot\delta_{n;(t,k)}]}.
\end{align*}
This last norm must be finite because of the continuity of the derivative $u_x^{(k)}$ on $\R_{>n}$. This shows that $E_n\subset [p_{n;(k,t)}<\infty]$, therefore $E_n$ satisfies assertion (ii) with respect to $P_n$.
\\
Finally, it remains to prove that $E_n$ satisfies assertion (iv) from Remark \ref{rem condition (iv)}. Take a sequence $(x_j)_{j\in\N}$ and assume that there is $x\in X$ such that $p_{n;(k,t)}(x-x_j)\to 0$ as $j\to \infty$ for all $k\in\N$ and all $t>n$. We prove by induction that $u_x$ is $k$-times differentiable on $\R_{>n}$ for all natural numbers $k$. Indeed, this implies that $x\in E_n$, which is the claim. \\
For $k=0$ there is nothing to do. Passing to the induction step, assume that $u_x$ is $k$-times differentiable on $\R_{>n}$. We have to show that $u_x$ is $(k+1)$-times differentiable on the same interval.
Fix $t\in\R_{>n}$, and let $f_j : I_{n; (t,k+1)}\rightarrow Y$, $j\in\N$, and $f : I_{n; (t,k+1)}\rightarrow Y$ given by
\[
f_j(s)\coloneqq \frac{u_{x_j}^{(k)}(t+s)-u_{x_j}^{(k)}(t)}{s}\,\,\, (j\in\N), \quad  \mbox{and} \quad f(s)\coloneqq \frac{u_{x}^{(k)}(t+s)-u_{x}^{(k)}(t)}{s}.
\]
Using \eqref{relation derivative and seminorms}, we get for all $s\in I_{n; (t,k+1)}$ the inequality

\begin{align*}
\norm{f(s)-f_j(s)} &= \Norm{\lim_{h_k\to 0}\ldots\lim_{h_1\to 0}\Big( D_{n;(k+1,t)}[x]\big(h_1,\ldots,h_k,s\big) \\
& \hspace{4cm}-D_{n;(k+1,t)}[x_j]\big(h_1,\ldots,h_k,s\big)\Big)} \\[5pt]
&=\lim_{h_k\to 0}\ldots\lim_{h_1\to 0} \Norm{D_{n;(k+1,t)}[x-x_j]\big(h_1,\ldots,h_k,s\big)} \\[10pt]
&\le p_{n; (k+1,t)}(x-x_j).
\end{align*}
Thus, the functions $f_j$ converge uniformly to $f$ for $j\to\infty$. Moreover, as $x_j\in E_n$ for each $j\in\N$, the limit 
$\lim_{s\to 0} f_j(s)$ exists for each $j\in \N$. Now, Lemma \ref{lemma interchange of limits} applies with $S=I_{n; (t,k+1)}$ and it guarantees the existence of the limit
$\lim_{s\to 0} f(s)$, which implies that $u_x$ is $(k+1)$-times differentiable at $t$. As $t$ is chosen arbitrarily, we deduce that $u_x$ is $(k+1)$-times differentiable on $\R_{>n}$, and this concludes the inductive proof.\\ 
Thus, $E_n$ is a $P$-subspace.
\end{proof}


\subsection{Real Analyticity}

For what is coming we follow \cite[Chapter 1]{KrPa}. Let $U\subset\R$ be open and let $X$ be a Banach space. A function $f:U\rightarrow X$ is said to be \textit{(real) analytic} at $t_0\in U$ if it can be represented by a convergent power series on some interval of positive radius centred at $t_0$ i.e., if there is $r\in\R_{>0}$ and a sequence $(x_k)_{k\in\N}$ in $X$ such that for all $t\in(t_0-r,t_0+r)$
\[
f(t)=\sum_{k=0}^\infty (t-t_0)^k x_k.
\] 
The function $f$ is said to be \textit{(real) analytic} on $W\subset U$ if it is analytic at each $t\in W$. The set of analytic functions on $W$ will be denoted by $\uC^\omega(W;X)$.

\begin{lemma} 
The following assertions hold:
\begin{enumerate}[label=(\roman*), font=\normalfont, noitemsep]
\item Let $(x_n)_{n\in\N}$ be a sequence in $X$, let $t_0\in\R$ and suppose that the power series 
\[
f(t)=\sum_{k=0}^\infty (t-t_0)^k x_k
\] 
has radius of convergence $R$. Then, for each $0<\rho<R$ there exists a constant $C_\rho>0$ such that for all $k\in\N_0$ and all $t\in (t_0-\rho,t_0+\rho)$ 
\[
\norm{f^{(k)}(t)} \le \frac{\rho\, C_\rho}{\rho-\abs{t-t_0}}\cdot \frac{k!}{\big(\rho-\abs{t-t_0}\big)^k}.
\]
\item Let $t_0\in\R$, $r>0$, and $f\in\uC^\infty\big((t_0-r,t_0+r);X\big)$. Assume that there is a constant $C>0$ such that for all $k\in\N$
\[
\norm{f^{(k)}(t)}\le C\cdot\frac{k!}{r^k} \qquad \mbox{ for all } t\in (t_0-r,t_0+r).
\]
Then $f\in \uC^\omega\big((t_0-r,t_0+r);X\big)$.
\end{enumerate}
\end{lemma}

\begin{proof}
Let us start with assertion (i). The fact that for each $0<\rho<R$ there is a constant $C_\rho>0$ such that \[
\norm{f^{(k)}(t_0)}\le C_\rho\cdot \frac{k!}{\rho^k}
\]
for all $k\in\N$ is well known, see for example the Cauchy estimates in \cite[p.~97]{HiPh}. At this point, as $x_k=\frac{f^{(k)}(t_0)}{k!}$, we compute for $t\in (t_0-\rho,t_0+\rho)$:
\[
\norm{f^{(k)}(t)}\le \frac{C_\rho}{\rho^k}\sum_{n=k}^\infty \frac{n!}{(n-k)!}\left(\frac{\abs{t-t_0}}{\rho}\right)^{n-k} = \frac{C_\rho}{\rho^k}\cdot \frac{k!}{\left(1-\frac{\abs{t-t_0}}{\rho}\right)^{k+1}},
\]
which is the claim. \\
Passing to assertion (ii), using the Weierstrass M-Test we see immediately that 
\[
\sum_{n=0}^\infty \frac{f^{(n)}(t_0)}{n!} (t-t_0)^n
\]
converges absolutely at least on $(t_0-r,t_0+r)$. Finally, Taylor's theorem for holomorphic vector-valued functions guarantees that this power series converges to $f$.
\end{proof}

From this lemma we obtain the next proposition, which characterizes analytic functions.

\begin{prop}\label{analytic vs derivatives}
Let $U\subset\R$ be an open set and $X$ a Banach space. A function $f:U\rightarrow X$ belongs to $\uC^\omega(U;X)$ if and only if $f$ belongs to $\uC^\infty(U;X)$ and its derivatives can be controlled in the following sense: for each $t\in U$ there are constants $C,r\in\R_{>0}$ such that $(t-r,t+r)\subset U$ and
\begin{align}\label{estimate derivatives}
\norm{f^{(k)}(s)}\le C\cdot\frac{k!}{r^k}\qquad\mbox{ for all } s\in (t-r,t+r) \mbox{ and } k\in\N.
\end{align}
\end{prop}

Proposition \ref{analytic vs derivatives} can be used to define suitable seminorms in order to obtain the following result.

\begin{prop}\label{ev analytic}
A function $u:\rpl\rightarrow\caL(X;Y)$ is individually eventually real analytic if and only if it is uniformly eventually real analytic.
\end{prop}

\begin{proof}
For each $n\in\N$ consider the subspace
\[
E_n\coloneqq \big\{x\in X \st u_x \mbox{ is real analytic on } \R_{>n}\big\}.
\]
In order to apply Corollary \ref{corolla}, we prove that each $E_n$ is a $P$-subspace. 
Since each orbit $u_x$ is real analytic on $(t_x,\infty)$, it is also smooth on the same interval by Proposition \ref{analytic vs derivatives}. Proposition \ref{propsmooth} implies that $u$ is eventually smooth from a starting point $t_0\in\rpl$, and, without loss of generality, we may assume $t_0=0$ (by passing to the shifted function $\rpl\rightarrow\caL(X;Y)$, $t\mapsto u(t_0+t)$ if necessary).
\\
Put $Q\coloneqq (0,1]\cap\Q$. For each $r\in Q$ and $t\in\R_{>1}$ define the seminorm
\[
p_{t,r}: X\rightarrow\rplc, \quad p_{t,r}(x)\coloneqq \sup\bigg\{\frac{r^k}{k!}\big\lVert u_x^{(k)}(s)\big\rVert\st s\in (t-r,t+r),\, k\in\N\bigg\}.
\]
Given $x\in X$, by Proposition \ref{analytic vs derivatives} the orbit $u_x$ is real analytic at $t$ if and only if there is $r\in Q$ such that $p_{t,r}(x)<\infty$. Thus,
\[
E_n=\bigcap_{t\in\R_{>n}}\bigcup_{r\in Q} [p_{t,r}(x)<\infty],
\]
which is assertion (ii) from Definition \ref{definition P subspace}, while assertion (i) is satisfied by Remark \ref{remark supremum continuous}. Finally, $E_n$ satisfies assertion (iii) trivially and therefore it is a $P$-subspace of $X$.
\end{proof}

We now lay out some consequences of Proposition \ref{ev analytic}. Firstly, we introduce some useful notation.

\begin{nota}\label{notation Minkowski sum}
For $z\in\C$ and $r\in\R_{>0}$ we denote the open ball centred at $z$ with radius $r$ by
\[
B(z,r)\coloneqq \{w\in\C\st \abs{w-z}<r\}.
\]
If $A_1,\ldots,A_n$ are subsets of $\C$ we write
\[
\sum_{j=1}^n A_j\coloneqq A_1+\ldots+A_n\coloneqq\Big\{\sum_{j=1}^n a_j \st a_1\in A_1,\ldots,a_n\in A_n\Big\}
\]
for the Minkowski sum of the sets $A_1,\ldots,A_n$.
\end{nota}

\begin{prop}\label{prop real analytic vs holoholo}
Let $V\subset\R$ be open and let $u:V\rightarrow\caL(X;Y)$ be a function such that $u_x\in\uC^\omega(V;Y)$ for every $x\in X$, where
\[
u_x:V\rightarrow Y, \qquad u_x(t)\coloneqq u(t)x \qquad\qquad (x\in X).
\]
Then there exists an open set $U\subset\C$ containing $V$ such that $u$ extends to $U$ to a holomorphic function. If $V$ is connected, then $U$ is a domain.
\end{prop}

\begin{proof}
Assume that for all $x\in X$ the orbits $u_x$ are functions in $\uC^\omega(V;Y)$. Then, given $t\in V$, for every $x\in X$ there is a positive radius $r_x\in\R_{>0}$ such that $u_x$ can be represented by a power series on $(t-r_x,t+r_x)$. 
By Proposition \ref{analytic vs derivatives}, $u_x\in\uC^\infty(V;Y)$ for all $x\in X$ and that the estimates \eqref{estimate derivatives} on the derivatives are satisfied on $(t-r_x,t+r_x)$ for all $x\in X$. Similarly as in the proof of Proposition \ref{ev analytic}, define for each $r\in Q\coloneqq (0,1]\cap\Q$ the extended seminorm
\[
p_{t, r}: X\rightarrow\rplc, \quad p_{t,r}(x)\coloneqq \sup\bigg\{\frac{r^k}{k!}\big\lVert u_x^{(k)}(s)\big\rVert\st s\in V\cap (t-r,t+r),\, k\in\N\bigg\}.
\]
Then $[p_{t,r}\le 1]$ is closed in $X$ by Remark \ref{remark supremum continuous} and
\[
\bigcup_{r\in Q} [p_{t,r}<\infty]=X \qquad (t\in V)
\]
since $u_x$ is real analytic at $t$ for every $x\in X$. Thus, as each subspace $[p_{t,r}<\infty]$ is a $P$-subspace, at least one of them is equal to $X$ by Proposition \ref{prop baire}. It follows that for some $r_t\in\R_{>0}$ the orbit $u_x$ can be represented by a power series centred at $t$ on the interval $(t-r_t,t+r_t)$ for every $x\in X$. The radius of convergence can be chosen independently of $x$ for every $t\in V$ and $u_x$ extends holomorphically to $B(t,r_t)$ for every $x\in X$. Finally, $u_x$ extends holomorphically to
\[
U\coloneqq \bigcup_{t\in V} B(t,r_t)
\]
for every $x\in X$, which implies that $u$ is strongly holomorphic on $U$ and therefore holomorphic on $U$. If $V$ is connected, then $U$ is clearly path-connected and hence connected and thus a domain.
\end{proof}

Proposition \ref{ev analytic} and Proposition \ref{prop real analytic vs holoholo} lead to the following obvious consequence.

\begin{corolla}
Let $u:\rpl\rightarrow\caL(X;Y)$ be an individually eventually real analytic function. Then there exists a time $t_0\in\rpl$ and a domain $U_0\subset\C$ with $\R_{>t_0}\subset U_0$ such that $u$ extends holomorphically to $U_0$. 
\end{corolla}

We say that a function $T:\rpl\rightarrow\caL(X)$ has the \textit{semigroup property} if 
\[
T(s+t)=T(s)T(t) \qquad \mbox{ for all } s,t\in \rpl. 
\]
In the case of individually eventually real analytic operator-valued functions having the semigroup property, we can say a little more.

\begin{lemma}\label{lemma composition of derivatives}
Let $T:\rpl\rightarrow\caL(X)$ have the semigroup property and assume that there is some $t_0\in\rpl$ such that $T_x\in\uC^\infty(\R_{>t_0};X)$ for all $x\in X$. Then, for all $s,t\in\R_{>t_0}$ and all $k,l\in\N_0$ 
\[
T^{(l)}(t)T^{(k)}(s)=T^{(l+k)}(t+s).
\]
\end{lemma}

All operators in the assertion of Lemma \ref{lemma composition of derivatives} exist by Lemma \ref{deriv lin op}. We stress that the derivatives of $T$ as operators in $\caL(X)$ should be understood with respect to the strong operator topology, cf.\ Remark \ref{remark strong vs norm derivative}.

\begin{rem}
When $T$ is a $C_0$-semigroup, then Lemma \ref{lemma composition of derivatives} follows easily from Proposition \ref{prop pazy}. In fact, if $A$ is the generator of $T$, then $T(t)X\subset\dom(A^n)$ for all $t\in\R_{> t_0}$ and all $n\in\N$, thus
\[
T^{(l)}(t)T^{(k)}(s)=A^lT(t)A^kT(s)=A^{l+k}T(t)T(s)=A^{l+k}T(s+t)=T^{(l+k)}(s+t).
\]
\end{rem}

\begin{proof}[Proof of Lemma \ref{lemma composition of derivatives}]
We make use of Lemma \ref{das lemma}, where all appearing limits should be understood to converge with respect the strong operator topology. For $\delta>0$ small enough
\[
T^{(k)}(s)=\lim_{h_k\to 0}\cdots\lim_{h_1\to 0} D_{\delta,k,s}[T](h_1,\ldots,h_k)
\]
and
\[
T^{(l)}(t)=\lim_{h_{k+l}\to 0}\cdots\lim_{h_{k+1}\to 0} D_{\delta,l,t}[T](h_{k+1},\ldots,h_{k+l}).
\]
Observe now that by the semigroup property
\begin{align*}
&D_{\delta,l,t}[T](h_{k+1},\ldots,h_{k+l})D_{\delta,k,s}[T](h_1,\ldots,h_k)  \\
&=\frac{1}{h_{k+1}\cdot\ldots\cdot h_{k+l}} \sum_{L\subset [l]} (-1)^{k-|L|}\, T(t+h_{k+L})
\frac{1}{h_1\cdot\ldots\cdot h_k} \sum_{K\subset [k]} (-1)^{k-|K|}\, T(s+h_K)\\
&=\frac{1}{h_{k+l}\cdots h_1}\sum_{L\subset [l]}\sum_{K\subset [k]} (-1)^{k+l-|K|-|L|}T(s+t+h_{k+L}+h_K)\\
&=\frac{1}{h_1\cdots h_{l+k}}\sum_{J\subset[l+k]}(-1)^{l+k-|J|}T(s+t+h_J)\\
&=D_{\delta,\,l+k,\,s+t}[T](h_1,\ldots,h_k,h_{k+1},\ldots,h_{k+l})
\end{align*}
Hence, by the continuity of the operators and by taking the product we have
\begin{align*}
T^{(l)}(t)&T^{(k)}(s) \\
&=\lim_{h_{k+l}\to 0}\cdots\lim_{h_{k+1}\to 0} D_{\delta,l,t}[T](h_{k+1},\ldots,h_{k+l})\lim_{h_k\to 0}\cdots\lim_{h_1\to 0} D_{\delta,k,s}[T](h_1,\ldots,h_k) \\
&=\lim_{h_{k+l}\to 0}\cdots\lim_{h_{k+1}\to 0}\lim_{h_k\to 0}\cdots\lim_{h_1\to 0} D_{\delta,l,t}[T](h_{k+1},\ldots,h_{k+l}) D_{\delta,k,s}[T](h_1,\ldots,h_k) \\
&= \lim_{h_{k+l}\to 0}\cdots\lim_{h_1\to 0}D_{\delta,\,l+k,\,s+t}[T](h_1,\ldots,h_k,h_{k+1},\ldots,h_{k+l})\\
&= T^{(l+k)}(t+s).
\end{align*}
\end{proof}

\begin{prop}\label{prop domain of holomorphy}
Let $T:\rpl\rightarrow\caL(X)$ have the semigroup property and assume that there is some $t_0\in\R_{>0}$ such that $T_x\in\uC^\omega(\R_{>t_0};X)$ for all $x\in X$. Then, for each $t\in\R_{>t_0}$, there exists an open domain $\Omega_t$ of $\C$ such that:
\begin{enumerate}[label=(\roman*), font=\normalfont, noitemsep]
\item $\R_{\ge t}\subset\Omega_t$ and $\{h\}+\Omega_t\subset\Omega_t$ for all $h\in\rpl$;
\item $(\Omega_t,+)$ is a semigroup;
\item there is an angle $\theta_t\in (0,\pi/2]$ and a sector of the form
\[
S(t,\theta_t)\coloneqq \big\{z\in\C\st \re z>t \mbox{ and } \abs{\arg z-t}<\theta_t\big\} \qquad \mbox{(as in Figure \ref{fig:settore})}
\]
which is contained in $\Omega_t$;
\item the semigroup $T$ can be extended to an holomorphic function $T:\Omega_t\rightarrow\caL(X)$ which conserves the semigroup property on $\Omega_t$, that is,
\[
T(z+w)=T(z)T(w) \qquad \mbox{ for all } \quad z,w\in\Omega_t\cup\rpl.
\] 
\end{enumerate}
\end{prop}

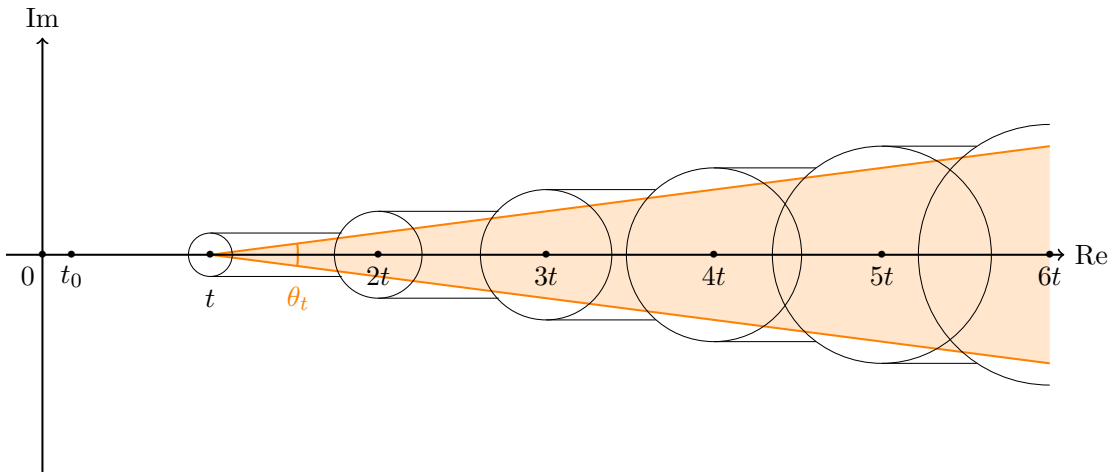
\begin{figure}[h]
\centering
\scalebox{0.96}{
\begin{tikzpicture}
\fill[orange!20!white] (2.3, 0)--(13.8,1.5)--(13.8, -1.5)--(2.3, 0);
 \draw[thick, orange]
    (13.8,-1.5) coordinate (a) 
    -- (2.3,0) coordinate (b) 
    -- (13.8,1.5) coordinate (c) 
    pic[draw=orange, -, angle eccentricity=1.2, angle radius=1.2cm]
    {angle=a--b--c};
    \draw[orange] node at (3.5, -0.6) {$\theta_t$};
\draw[->, thick] (-0.5,0)--(14,0) node[right]{$\re$};
\draw[->, thick] (0,-3)--(0,3) node[above]{$\im$};
\draw[] node at (-0.2, -0.3) {$0$};
\draw[] node at (0, 0) {\tiny\textbullet};
\draw[] node at (0.4, 0) {\tiny\textbullet};
\draw[] node at (0.4, -0.3) {$t_0$};
\draw[] node at (2.3, 0) {\tiny\textbullet};
\draw[] node at (4.6, 0) {\tiny\textbullet};
\draw[] node at (6.9, 0) {\tiny\textbullet};
\draw[] node at (9.2, 0) {\tiny\textbullet};
\draw[] node at (11.5, 0) {\tiny\textbullet};
\draw[] node at (13.8, 0) {\tiny\textbullet};
\draw[] node at (2.3, -0.6) {$t$};
\draw[] node at (4.6, -0.3) {$2t$};
\draw[] node at (6.9, -0.3) {$3t$};
\draw[] node at (9.2, -0.3) {$4t$};
\draw[] node at (11.5, -0.3) {$5t$};
\draw[] node at (13.8, -0.3) {$6t$};
%
%
\draw (2.3,0) circle (0.3cm);
\draw (4.6,0) circle (0.6cm);
\draw (6.9,0) circle (0.9cm);
\draw (9.2,0) circle (1.2cm);
\draw (11.5,0) circle (1.5cm);
\draw (13.8,1.8) arc (90:270:1.8);
\draw (2.3,0.3) -- (4.1,0.3);
\draw (2.3,-0.3) -- (4.1,-0.3);
\draw (4.6,0.6) -- (6.25,0.6);
\draw (4.6,-0.6) -- (6.25,-0.6);
\draw (6.9,0.9) -- (8.4,0.9);
\draw (6.9,-0.9) -- (8.4,-0.9);
\draw (9.2,1.2) -- (10.6,1.2);
\draw (9.2,-1.2) -- (10.6,-1.2);
\draw (11.5,1.5) -- (12.8,1.5);
\draw (11.5,-1.5) -- (12.8,-1.5);
%
%
\end{tikzpicture}
}
\caption{Domain $\Omega_t$ with sector $S(t,\theta_t)$.}
\label{fig:settore}
\end{figure}

\begin{proof}[Proof of Proposition \ref{prop domain of holomorphy}]
Since $T$ is holomorphic at $t$, it is possible to represent $T$ as a power series having radius of convergence $r$ for some $r\in (0,t)$. This implies that
\begin{align}\label{power series}
T(z)=\sum_{k=0}^\infty (z-t)^k\frac{T^{(k)}(t)}{k!} \qquad \mbox{ for all } z\in B(t,r).
\end{align}
We start by showing that we can choose
\[
\Omega_t\coloneqq \bigcup_{n\in\N}\,\,\,\,\,\bigcup_{nt\,\le\, s\,<\,(n+1)t} B(s,nr),
\]
(see also Figure \ref{fig:settore}). In this case, assertions (i) is satisfied by construction and assertion (ii) follows from the fact that, for each $s\in\R_{\ge t}$ and $n\in\N$,
\[
B(ns,nr)=\sum_{j=1}^n B(s,r).
\]
By Lemma \ref{lemma composition of derivatives}, the equality
\[
T^{(k)}(s)=T^{(k)}(t)T(s-t)=T(s-t)T^{(k)}(t)
\]
holds true for all $s\in\R_{\ge t}$ and $k\in\N_0$. Thus, since 
\[
\Big(\sum_{k=0}^\infty (z-s)^k\frac{T^{(k)}(t)}{k!}\Big)T(s-t)=\sum_{k=0}^\infty (z-s)^k\frac{T^{(k)}(t)T(s-t)}{k!}=\sum_{k=0}^\infty (z-s)^k\frac{T^{(k)}(s)}{k!},
\]
the series on the right-hand side has radius of convergence at least equal to $r$, which shows that $T$ can be extended holomorphically to $B(s,r)$. The same reasoning shows that $T$ can be extended holomorphically on $\bigcup_{s\ge t} B(s,r)$. 
We now have to show that $T$ can be extended holomorphically to $B(nt,nr)$ for every $n\in\N$. This follows from the next computation, which is based on Cauchy product of power series and Merten's theorem. For $z\in(t-r,t+r)$ the series \eqref{power series} converges absolutely, hence
\begin{equation*}
\begin{aligned}[t]
T(nz)&= \overbrace{T(z) \cdots T(z)}^{n \mbox{ times}}=\prod_{j=1}^n\left(\sum_{k_j=0}^\infty (z-t)^{k_j}\frac{T^{(k_j)}(t)}{k_j!}\right) \\
&=\sum_{k_1=0}^\infty\sum_{k_2=0}^{k_1}\cdots\sum_{k_n=0}^{k_{n-1}}\left((z-t)^{k_n}\frac{T^{(k_n)}(t)}{k_n!}\right)\cdot\left((z-t)^{k_{n-1}-k_n}\frac{T^{(k_{n-1}-k_n)}(t)}{(k_{n-1}-k_n)!}\right)\cdots\\
\MoveEqLeft[-22]\cdots\left((z-t)^{k_{1}-k_2}\frac{T^{(k_1-k_2)}(t)}{(k_1-k_2)!}\right)\\
&=\sum_{k_1=0}^\infty(z-t)^{k_1}T^{(k_1)}(nt)\sum_{k_2=0}^{k_1}\cdots\sum_{k_n=0}^{k_{n-1}}\frac{1}{k_n!(k_{n-1}-k_n)!\cdots(k_1-k_2)!}\\
&=\sum_{k_1=0}^\infty(z-t)^{k_1}T^{(k_1)}(nt)\frac{n^{k_1}}{k_1!}
=\sum_{k=0}^\infty (nz-tn)^k\frac{T^{(k)}(nt)}{k!},
\end{aligned}
\end{equation*}
where we used Lemma \ref{lemma composition of derivatives} and the binomial theorem. The last series, which is centred at $nt$, has radius of convergence equal at least to $nr$, therefore $T$ extends holomorphically to $B(nt,nr)$. This, combined with what shown above, implies that $T$ extends holomorphically on $\Omega_t$. 
At this point, $S(t,\theta_t)\subset \Omega_t$ by choosing $\theta_t=\arcsin(r/t)$, which proves assertion (iii). It remains to prove assertion (iv). Firstly, take $h\in\rpl$ and observe that the maps $\Omega_t\rightarrow \caL(X)$ given by
\[
z\mapsto T(h)T(z) \qquad\mbox{ and }\qquad z\mapsto T(h+z)
\]
are both holomorphic and coincide on $\R_{\ge t}$, and consequently on the entire $\Omega_t$ by the identity theorem for holomorphic functions. Hence, $T(h)T(z)=T(h+z)$ whenever $z\in\Omega_t$ and $h\in\rpl$. Fix now some $w\in\Omega_t$ and consider the maps $\Omega_t\rightarrow \caL(X)$ given by
\[
z\mapsto T(w)T(z) \qquad\mbox{ and }\qquad z\mapsto T(w+z).
\]
These maps are holomorphic as well and they coincide on $\R_{\ge t}$. Using the identity theorem a second time, the claimed functional equation in assertion (iv) follows.
\end{proof}

An immediate consequence of Proposition \ref{ev analytic} and Proposition \ref{prop domain of holomorphy} is the following result.

\begin{corolla}
Let $T:\rpl\rightarrow\caL(X)$ be an individually eventually real analytic function. If $T$ has the semigroup property, then there exist some $t\ge 0$ and $\theta\in(0,\pi/2]$ such that $T$ extends holomorphically to the sector $S(t,\theta)$.
\end{corolla}

\section*{Acknowledgement} The author expresses his gratitude to Ljudevit Palle for introducing him to the \textit{extended seminorms} during a very inspiring discussion about the differentiability case. It should be stressed that the idea to interpret the difference quotient as an extended seminorm is due to him. The author thanks his supervisor Markus Haase and his colleague Alexander Dobrick for their numerous comments and corrections, which improved strongly both the introduction and the readability and the structure of this article overall. The author is grateful to Rainer Nagel for giving him the possibility of spending a week in Tübingen to discuss the content of this article in its very preliminary stage. The author wants to thank Henrik Kreidler, Florian Pannasch, Sascha Trostorff and Marcin Wnuk. Their thorough reports and different suggestions have led to several improvements of the article. At last, a special thanks goes to Jochen Glück for his constant encouragement and enthusiasm.

\bibliographystyle{plain}
\bibliography{literature}

\end{document}